\documentclass[12pt, a4paper]{amsart}
\usepackage{amssymb,fullpage}
\usepackage{amsmath,esint}
\usepackage[colorlinks, allcolors=RedViolet,pdfstartview=,pdfpagemode=UseNone]{hyperref}
\usepackage{colortbl}
\usepackage[dvipsnames]{xcolor}

\newtheorem{theorem}[equation]{Theorem}
\newtheorem{lemma}[equation]{Lemma}
\newtheorem{proposition}[equation]{Proposition}
\newtheorem{corollary}[equation]{Corollary}
\theoremstyle{definition}
\newtheorem{definition}[equation]{Definition}

\theoremstyle{remark}
\newtheorem{remark}[equation]{Remark}

\numberwithin{equation}{section}

\newcommand{\supp}{\operatorname{supp}}

\newcommand{ \R }{ \mathbb{R} }
\newcommand{ \Rn }{ {\mathbb{R}^n} }

\newcommand{\tphi}{{\tilde\phi}}
\newcommand{\diam}{\operatorname{diam}}
\newcommand{\loc}{{\operatorname{loc}}}
\renewcommand{\epsilon}{\varepsilon}
\renewcommand{\phi}{\varphi}
\renewcommand{\le}{\leqslant}
\renewcommand{\ge}{\geqslant}
\renewcommand{\leq}{\leqslant}
\renewcommand{\geq}{\geqslant}
\renewcommand{\div}{\operatorname{div}}

\newcommand{\ainc}[1]{\hyperref[ainc]{{\normalfont(aInc){\ensuremath{_{#1}}}}}}
\newcommand{\adec}[1]{\hyperref[adec]{{\normalfont(aDec){\ensuremath{_{#1}}}}}}
\newcommand{\inc}[1]{\hyperref[inc]{{\normalfont(Inc){\ensuremath{_{#1}}}}}}
\newcommand{\dec}[1]{\hyperref[dec]{{\normalfont(Dec){\ensuremath{_{#1}}}}}}
\newcommand{\azero}{\hyperref[azero]{{\normalfont(A0)}}}
\newcommand{\aone}{\hyperref[aone]{{\normalfont(A1)}}}

\newcommand{\wMA}{\hyperref[wVA1]{{\normalfont(wVA1)}}}
\newcommand{\MA}{\hyperref[VA1]{{\normalfont(VA1)}}}

\newcommand{\Phiw}{\Phi_{\rm{w}}}
\newcommand{\Phic}{\Phi_{\rm{c}}}

\begin{document}

\title{Maximal regularity for local minimizers of non-autonomous functionals}

\author{Peter H\"ast\"o}

\address{Department of Mathematics and Statistics, FI-20014 University of Turku, Finland 
and
Department of Mathematics, FI-90014 University of Oulu, Finland}
\email{peter.hasto@oulu.fi / peter.hasto@utu.fi}

\author{Jihoon Ok}
\address{Department of Mathematics, Sogang University, Seoul 04107, Republic of Korea}
\email{jihoonok@sogang.ac.kr}

\thanks{}





\begin{abstract}
We establish local $C^{1,\alpha}$-regularity for some $\alpha\in(0,1)$ and $C^{\alpha}$-regularity for any $\alpha\in(0,1)$ of local minimizers of the functional
\[
v\ \mapsto\ \int_\Omega \phi(x,|Dv|)\,dx,
\]
where $\phi$ satisfies a $(p,q)$-growth condition.
Establishing such a regularity theory with sharp, general conditions has been 
an open problem since the 1980s. 
In contrast to previous results, we formulate the continuity 
requirement on $\phi$ in terms of a single condition for the 
map $(x,t)\mapsto \phi(x,t)$, rather than separately in the 
$x$- and $t$-directions. 
Thus we can obtain regularity results for functionals without assuming that the gap $\frac qp$
between the upper and lower growth bounds is close to $1$.
Moreover, for $\phi(x,t)$ with particular structure, including $p$-, Orlicz-, 
$p(x)$- and double phase-growth, our single condition implies known, essentially optimal, regularity conditions. Hence, we handle regularity theory for the above functional in a universal way. 
\end{abstract}

\maketitle


\section{Introduction}\label{sect:intro}

The calculus of variations is a classical and still active topic in mathematics which is connected not only to other mathematical fields 
(partial differential equations, geometry, \dots)\ and but also to applications 
(physics, engineering, economy, \dots).
Research on regularity of minimizers of the functional
\[
v\ \ \mapsto \ \ \mathcal F(v,\Omega):=\int_\Omega F(x,Dv)\,dx
\]
has been a major topic in calculus of variations and PDEs. 
If $F$ depends only on the gradient, i.e.\ $F(x,z)\equiv F(z)$, $\mathcal F$ 
is called an \textit{autonomous} functional. 
The simplest non-linear model case is the $p$-power function
\[
F(z)=|z|^p,\quad 1<p<\infty.
\]
The corresponding Euler-Lagrange equation is the $p$-Laplace equation
$\div(|Du|^{p-2}Du)=0$, 
and the maximal regularity of weak solutions of $p$-Laplace equations is $C^{1,\alpha}$ for some $\alpha\in(0,1)$ depending only on $p$ and the dimension $n$. 
We refer to \cite{AF89,DiBe1,Eva3,Le1,Man86,To84,Uhl77,Ura68,Wa93,Wa94} for 
classical results on $C^{1,\alpha}$-regularity for equations and systems of $p$-Laplacian type.

On the other hand, if $F$ depends on both the space variable and the gradient, 
$\mathcal F$ is called a \textit{non-autonomous} 
functional, and this has been a central topic in contemporary regularity theory. 
The main approach to such minimization problems is due to 
Giaquinta and Giusti \cite{GiaG83,GiaG84}. It is based on the 
following $p$-type growth conditions:
\[
\begin{cases}
z\mapsto F(x,z)\text{ is } C^2, \\
\nu |z|^p \le F(x,z)\le L(1+|z|^p), \\
\nu (\mu^2+|z|)^{\frac{p-2}{2}}|\lambda|^2 
\le F_{zz}(x,z)\lambda\cdot \lambda 
\le L (\mu^2+|z|^2)^{\frac{p-2}{2}} |\lambda|^2, \\
|F(x,z)-F(y,z)| \le \omega(|x-y|)(1+|z|^p). 
\end{cases} 
\]
This essentially corresponds to the perturbed case $a(x)|z|^p$
with the same $p$-type growth assumed at all points. 
Lieberman \cite{Lie1} extended this to the case where $|z|^p$ is replaced 
by $\phi(|z|)$. 
However, such structure conditions fail to accommodate many kinds of energy functionals 
since the variability in the $x$- and $z$-directions are treated separately. 

The need to treat the $x$- and $z$-directions separately 
leads Mingione to conclude in his influential survey that 
``regularity results should be chased [in more general cases] by looking at 
special classes of functionals and
thinking of relevant model examples, thereby limiting the degree of generality one
wants to achieve'' \cite[p.~405]{Min06}.
In this spirit, the most significant non-autonomous functionals in the literature 
have so-called Uhlenbeck structure, i.e.\ $F$ depends on $t:=|z|$ instead of $z$, 
\[
F(x,z)=\phi(x,|z|)=\phi(x,t),
\]
and are the following: 
\begin{enumerate}
\item[I.]
Perturbed Orlicz: $a(x) \psi(t)$, where $0<\nu\leq a(\cdot)\leq L$ and 
$\psi'(t)\approx t \psi''(t)$.
\item[II.]
Variable exponent: $t^{p(x)}$, where $1<p^-\leq p(\cdot) \leq p^+<\infty$.
\item[III.]
Double phase: $t^p + a(x) t^q$, where $1<p\leq q$ and $a(\cdot)\geq0$. 
\end{enumerate}
These models were first studied by Zhikov \cite{Zhi86, Zhi95} in the 1980's in relation to 
Lavrentiev's phenomenon and have been considered in hundreds of papers since \cite{Min06, Rad15}.
In keeping with Mingione's thesis, regularity results for these cases have been 
established in independent, idiosyncratic ways (cf. Section~\ref{sect:literature}). 
Moreover, various variants and borderline cases have been investigated, such as:
\begin{itemize}
\item[IV.]
Perturbed variable exponent: $t^{p(x)} \log(e+t)$, e.g.\ \cite{GP13,LiaSZ18,Ok0,Ok2}.
\item[V.]
Orlicz variable exponent: $[\psi(t)]^{p(x)}$ or $\psi(t^{p(x)})$, e.g.\ \cite{CapCF18, GPRT17}.
\item[VI.]
Degenerate double phase: $t^p + a(x) t^p \log (e+t)$, e.g.\ \cite{BCM2,BOh2}.
\item[VII.]
Orlicz double phase: $\psi(t) + a(x) \xi(t)$, e.g.\ \cite{BOh3}.
\item[VIII.]
Triple phase: $t^p + a(x) t^q + b(x) t^r$, e.g.\ \cite{DeFOh1, FanZZ_pp}. 
\item[IX.]
Double variable exponent: $t^{p(x)}+t^{q(x)}$, e.g.\ \cite{CenRR18, RadRSZ_pp,ZhaR18}.
\item[X.]
Variable exponent double phase $t^{p(x)} + a(x) t^{q(x)}$, e.g.\ \cite{MaeMOS_pp, RagT20}. 
\end{itemize}

In this paper, we establish a general regularity theory for non-autonomous functionals
with Uhlenbeck structure based on a single condition involving both the $x$- and $t$-directions. 
Specifically, we prove 
maximal local regularity properties, i.e.\ $C^{1,\alpha}$-regularity for 
some $\alpha\in(0,1)$ and $C^{\alpha}$-regularity for any $\alpha\in (0,1)$.
We consider a convex function $\phi:\Omega\times[0,\infty)\to[0,\infty)$ 
 satisfying the following ``vanishing A1''  variant of 
\aone{} (see Definitions~\ref{defPhi} and \ref{Def41}, below):
\begin{itemize}
\item[\normalfont{(VA1)}] 
There exists a non-decreasing continuous function 
$\omega:[0,\infty)\to[0,1]$ with $\omega(0)=0$ 
such that for any small ball $B_r\Subset \Omega$,
\[
\phi^+_{B_r}(t)\leq (1+\omega(r))\phi^-_{B_r}(t)\quad\text{for all }\ t>0\ \ 
\text{satisfying }\ \phi^-_{B_r}(t)\in[\omega(r),|B_r|^{-1}],
\]
\end{itemize}
where $\phi^+_{B_r}(t)$ and $\phi^-_{B_r}(t)$ are the supremum and infimum 
of $\phi(\cdot,t)$ in $B_r$, respectively.  
Let us point out that \MA{} is optimal for Theorem~\ref{mainthm1-0} in the following sense: 
For any $\theta<1$ assume that \MA{} is replaced by 
\[
\phi^+_{B_r}(t)\leq (1+\omega(r))\phi^-_{B_r}(t)\quad\text{for all }\ t>0\ \ 
\text{satisfying }\ \phi^-_{B_r}(t)\in[\omega(r),|B_r|^{-\theta}].
\]
Then the conclusions of the theorem do not hold, as is shown by examples in \cite{Min06} 
already in the double phase case (cf.\ Corollary~\ref{corexdouble2}), see also 
\cite{BalDS_pp, BenHHK_pp}.  Furthermore, $1+\omega(r)$ in the inequality from \MA{} ensures the continuity of 
the function, which is necessary already in the perturbed linear case 
(cf. Corollary~\ref{cor:PerturbedCase} and Remarks~\ref{rmk:PerturbedCase1} and \ref{rmk:PerturbedCase2}).

\begin{theorem}\label{mainthm1-0}
Let $\phi\in \Phiw(\Omega)$, $\phi(x,\cdot)\in C^1([0,\infty))$ for every $x\in\Omega$ 
with $\partial_t\phi$ satisfying \azero{}, \inc{p-1} and \dec{q-1} for some $1<p\leq q$ and let $u\in W^{1,\phi}_{\loc}(\Omega)$ be a local minimizer of the 
$\phi$-energy
\begin{equation}\label {mainfunctional}
\int_\Omega \phi(x,|\nabla u|)\, dx.
\end{equation}
\begin{enumerate}
\item
If $\phi$ satisfies \MA, then $u\in C^{\alpha}_{\loc}(\Omega)$ for any $\alpha\in(0,1)$.
\item
If $\phi$ satisfies \MA{} and $\omega(r)\le c r^\beta$ for some $c,\beta>0$, 
then $u\in C^{1,\alpha}_{\loc}(\Omega)$ for some $\alpha\in(0,1)$. Here $\alpha$ depends only on $n,p,q,L$ and $\beta$, where $L\geq1$ is from \azero. 
\end{enumerate}
\end{theorem} 

\begin{remark}\label{rmk:PerturbedCase1}
In this paper, we consider $\varphi(x,t)$ continuous in $x$. It is clear that we cannot remove the 
assumption $\lim_{r\to0} \omega(r)=0$ from \MA{} and still obtain $C^\alpha$-regularity 
for all $\alpha\in(0,1)$. However, continuity is not strictly speaking necessary, as it is known 
for $\phi(x,t)=a(x)\psi(t)$ with $a$ locally $\mathrm{VMO}$ (vanishing mean oscillation), that the corresponding minimizer 
is in $C^{\alpha}_{\mathrm{loc}}$ for any $\alpha\in(0,1)$, in fact, in $W^{1,p}_{\mathrm{loc}}$ for any $p>1$.
It seems that for this result the special multiplicative structure is important.
\end{remark}

\begin{remark}\label{rmk:PerturbedCase2}
If we consider solutions of the general linear elliptic equation $\mathrm{div}(A(x)Du)=0$,
where $A(x)$ is a bounded and uniformly elliptic $n\times n$ matrix, 
then the continuity of $A$ does not imply that the function is Lipschitz or its derivative is continuous 
\cite[Propositions~1.5 and 1.6]{JinMS09}. Therefore, we cannot expect to remove the 
assumption $\omega(r)\le cr^\beta$ from \MA{} and still obtain $C^{1,\alpha}$-regularity. 
\end{remark}

We shall introduce notation, assumptions and properties of generalized $\Phi$-functions and 
related spaces later in Section~\ref{sect:preliminary}. Recall that \textit{local minimizer} means that $u$ satisfies 
\[
\int_{\Omega'} \phi(x,|\nabla u|)\, dx \leq \int_{\Omega'} \phi(x,|\nabla v|)\, dx
\]
for every $v\in W^{1,\phi}(\Omega')$ with $u-v\in W^{1,\phi}_0(\Omega')$ and $\Omega'\Subset \Omega$. 

In fact, we will generalize \MA{} to a weaker version, \wMA{}, which covers not only 
\MA{} but its borderline cases (see Remark~\ref{rmkMA1}) as well as the PDE case 
(see Remark~\ref{rem:pdeCase}), and under this condition we will prove 
$C^{\alpha}$- and $C^{1,\alpha}$-regularity, see Theorems~\ref{mainthm1} 
and \ref{mainthm2}. As far as we know, 
these theorems cover all previously known results (and several new ones) of 
$C^{\alpha}$- or $C^{1,\alpha}$-regularity for the functionals I--X (see Section \ref{sect:examples}) 
with the exception of VMO coefficients (Remark~\ref{rmk:PerturbedCase1}).

Even in the case of autonomous functionals (i.e.\ $\phi(x,t)\equiv \phi(t)$), our 
results provide slight extensions to the state-of-the-art. Up to now, maximal
regularity for autonomous functionals has been established assuming 
$\phi\in C^1([0,\infty))\cap C^2((0,\infty))$. However, in this paper we only assume $\phi\in C^1([0,\infty))$, that is, we do not assume that $\phi$ is twice differentiable. For instance, $\phi(t):= \int_0^t \min\{s,s^2\}\, ds$ (cf.\ \cite{AzzAP14}) is covered by our result but 
is not $C^2$.

\medskip

Let us conclude the introduction by outlining the approach of the paper and pointing out the main difficulties and innovations. 

The first difficulty for a reasonable regularity theory is to find a well-designed condition for general $\phi$. The regularity conditions on $\phi$ for the 
types I--III seem unconnected to one another, since in these cases, the behaviors of $\phi$ with respect to $x$ and $t$ can be investigated separately. 
Recently, on the other hand,
the $C^{\alpha}$-continuity with some small $\alpha>0$ for (quasi-)minimizers of the general non-autonomous functional has been established under the 
so-called \aone{} condition \cite{BenHHK_pp,HarHL_pp18, HarHT17}:
\begin{equation}\label{a1}
\phi^+_{B_r}(t)\leq L\phi^-_{B_r}(t)\quad\text{for all }\ t>0\ \ 
\text{satisfying }\ \phi^-_{B_r}(t)\in[1,|B_r|^{-1}].
\end{equation}
From this, it is natural to require $L\to 1$ as $r\to 0$ for higher regularity. 
Additionally, small values $t\le 1$ were previously lumped into an additive 
constant using decay at infinity. A more precise estimate, on the other hand, 
requires the previous condition to be extended from $[1,|B_r|^{-1}]$ 
to $[\omega(r),|B_r|^{-1}]$. 

The main difficulty is to find a suitably regular auxiliary autonomous function $\tphi(t)$ for the perturbation technique in which one approximates the minimizer with the solution to a related but simpler minimization problem.
In order for the perturbation argument to work under the assumption  \MA{}, the autonomous function $\tphi(t)$ should satisfy the following requirements:
\begin{itemize}
\item[(1)] $\tphi\in C^1([0,\infty))\cap C^2((0,\infty))$ and $t\tphi''(t)\approx \tphi'(t)$.
\item[(2)] For a given $B_r$ with small $r\in(0,1)$, $\tphi(t)$ is sufficiently close in some sense to $\phi(x,t)$ for all $(x,t)\in B_r\times [t_1,t_2]$,  where $t_1:=(\phi^-_ {B_r})^{-1}(\omega(r))$ and $t_2:=(\phi^-_ {B_r})^{-1}(|B_r|^{-1})$.
\item[(3)] $\theta_0(x,t):=\phi(x,\tphi^{-1}(t))$ satisfies \azero{}, \ainc{1}, \adec{q/p} and \aone{}. 
\end{itemize}
The construction of such $\tphi$ is quite nontrivial, since the property (3) is not satisfied in general for either 
$\tphi(t)=\phi(y,t)$ with any choice of $y\in \overline{B_r}$
or $\tphi(t)=\phi_{B_r}^-(t)$ (the expected choices based on previous research). 
Note that for type II (variable exponent) or type III (double phase),
one can simply take $\tphi(t)=t^{p_r}$ or $\tphi(t)=t^p+a_r t^q$, where 
$p_r:=\inf_{B_r}p(\cdot)$ and $a_r:=\inf_{B_r}a(\cdot)$, 
so this provides no guidance for the general case: 
 in these special cases $t\mapsto \phi(x,\tphi^{-1}(t))$ satisfies \ainc{1} since a single 
point captures the slowest growth for all values of $t$, whereas in general the slowest growth 
may occur at different locations for different $t$.

The requirements (1)--(3) above are crucially used in our comparison step. Let $v$ be a minimizer of an autonomous functional with $\tphi$-energy in $B_r$ satisfying $v=u$ on $\partial B_r$. Then by (1) and known regularity results for Orlicz growth, we obtain that $v$ is locally $C^{1,\alpha_0}$ for some $\alpha_0\in(0,1)$ (Lemma~\ref{lemHolder}). 
Moreover, from (3) we can deduce a global nonlinear Calder\'on--Zygmund type estimate 
in the generalized Orlicz space 
$L^{\theta}$ with $\theta=\theta_0^{1+\sigma_0}$ for some $\sigma_0>0$
(Lemma~\ref{lemCZ}), which implies that $Dv\in L^{\phi}(B_r)$ and so, with this $v$, we can use the minimizing property 
of $u$. 
Note that this approach is new even for the double phase problem, type III.

The Calder\'on--Zygmund type estimates (Lemma~\ref{lemCZ}) in generalized Orlicz space $L^\theta$ 
for the norm will be obtained by an extrapolation argument \cite{CruHas18}
and in this process \aone{} of $\theta$ suffices. However, we need a mean integral version of 
Calder\'on--Zygmund type estimate that is stable under the size of underlying domain and here \aone{} of $\theta$ is not enough. 
We overcome this problem by replacing $\theta(x,t)$ with $\theta(x,t)+t^{p_1}$ for suitable $p_1>1$ along with 
delicate analysis. Note that $\theta(x,t)+t^{p_1}$ satisfies a stronger assumption than \aone{}. As a consequence, 
there is ``$+1$'' in the mean integral version of estimate \eqref{meanCZestimate}. 

We construct our approximation $\tphi$ and derive the comparison estimate for $\phi$ and $\tphi$ in Section~\ref{sect:comparison}. In Proposition~\ref{prop:thetaOK} 
we show that our approximation satisfies the assumptions in (3), above, and in this 
step a new framework for generalized Orlicz spaces from \cite{HarH_book} is rather crucial. 
Then a comparison argument along with (2) and a higher integrability result for 
$Du$ yield that $Du$ is sufficiently close to $Dv$ in the mean oscillation sense
(Corollary~\ref{corcom}). 

We present proofs of some regularity results for autonomous problems in Appendices~\ref{app:a} 
and \ref{app:b}. We start this article with an overview of regularity 
theory in the $(p,q)$-growth case (Section~\ref{sect:literature}) and with 
notation and background (Section~\ref{sect:preliminary}).
 
\begin{remark}
Constructing a suitable $\tphi$ is the main 
problem also in extending this approach to the case without Uhlenbeck structure, i.e.\ 
energy functionals depending on the derivative $Du$, not just its norm. 
Namely, an approximation $\tphi:\Omega\times \Rn\to \R$ affords us much less 
room to operate in than $\tphi:\Omega\times [0,\infty) \to \R$. Indeed, it is not 
even clear how to state the appropriate assumptions in this case. In addition, 
the main tools from \cite{HarH_book} concern only the isotropic case 
$\phi(x,|Du|)$. Therefore, the regularity of the anisotropic minimization 
problem $\int \phi(x,Du)\, dx$ remains a question for future research.
\end{remark}

\begin{remark} 
The vectorial case, i.e. $u:\Omega\to \R^N$ with $N>1$, is also an interesting issue. 
The main difficulty in this case is the following: in order that the local minimizer of the regular autonomous functional 
with Orlicz function $\tphi=\tphi(t)$ have $C^{1,\alpha}$-regularity $\tphi$ should apparently 
satisfy not only $t\tphi''(t)\approx \tphi'(t)$ but also a H\"older type vanishing condition on $\tphi''$, see 
\cite[Assumption~2.2]{DieSV17}. It is unclear whether \MA{} or some modification implies the additional condition of 
$\tphi$. This is also a future research topic.    
\end{remark}


\section{Overview of regularity for \texorpdfstring{$(p,q)$}{(p,q)}-growth and special cases}
\label{sect:literature}

An alternative extension to the approach of Giaquinta and Giusti 
is to consider different upper and lower growth rates, and 
replace the exponent on the right-hand side by $q>p$. This leads to so-called 
$(p,q)$-growth functionals, for instance with assumptions 
\[
\begin{cases}
z\mapsto F(x,z)\text{ is } C^2, \\
\nu |z|^p \le F(x,z)\le L(1+|z|^q), \\
\nu (1+|z|)^{\frac{p-2}{2}}|\lambda|^2 
\le F_{zz}(x,z)\lambda\cdot \lambda 
\le L (1+|z|^2)^{\frac{q-2}{2}} |\lambda|^2, \\
|F(x,z)-F(y,z)| \le \omega(|x-y|) (1+|z|^q). 
\end{cases} 
\]
This case was introduced and systematically studied by Marcellini 
\cite{Mar89, Mar91, Mar93, Mar96a, Mar96b}. Several other researchers also 
contributed to the theory, cf.\ \cite{BecM_pp,EleMM_pp,Min06}. 
For instance, Marcellini \cite{Mar91} started by showing that that every minimizer in 
$W^{1,q}_\loc(\Omega)$ has locally bounded gradient provided $2\le p\le q$ and 
\[
\frac qp \le 1 + \frac 2{n-2}, \quad\text{when } n>2;
\] 
(the proof uses PDE techniques and entails several additional assumptions, 
which are not presented here; see also a recent improvement in \cite{BelS_pp}). 
Note, however, that $W^{1,q}_\loc(\Omega)$ is already higher integrability, 
so this is not a natural assumption in this context and was addressed in 
\cite[Section~3]{Mar91}. Later, 
Esposito, Leonetti and Mingione \cite{EspLM04} showed that every 
minimizer in $W^{1,p}_\loc(\Omega)$ also belongs to $W^{1,q}_\loc(\Omega)$, 
but only when 
\[
\frac qp \le 1+ \frac\beta n
\quad\text{for}\quad \omega\in C^\beta.
\]
Furthermore, they provide an example showing that if the latter condition does not 
hold, then a minimizer in $W^{1,p}_\loc(\Omega)$ need not belong to 
$W^{1,q}_\loc(\Omega)$ so the Lavrentiev phenomenon occurs. 

It seems that $(p,q)$-growth is the most general class of 
non-autonomous functionals in the calculus of variations. 
Regularity theory, including $C^{\alpha}$- and $C^{1,\alpha}$-regularity, 
in this general class is not easily obtained from classical regularity theory for functionals with standard $p$-growth, see for instance \cite{Min06}.
Furthermore, there are no general results in the $(p,q)$-case 
which cover the special cases I--X, so in that sense the theory is incomplete. 
We note that some recent papers 
\cite{BenHHK_pp,ChlGZ18, ChlGZ_pp19, HarHL_pp18, HarHT17, WanLZ19}
deal with calculus of variation in generalized Orlicz spaces, but these papers 
do not cover higher regularity.

Indeed, the $C^{\alpha}$- and $C^{1,\alpha}$-regularity theories for type I--III functionals 
have been proved in independent ways. For I, $\phi$ is nothing but an autonomous functional with coefficient, and so regularity results can be obtained by using a standard perturbation argument. On the other hand, II and III are quite different from I, since 
they are potentially non-uniformly elliptic problems. Formally, we can rewrite the energy functions as
$$
\text{II: } |Du|^{p(x)-p^-} |Du|^{p^-} 
\quad\text{and}\quad 
\text{III: } (1+a(x)|Du|^{q-p}) |Du|^{p}.
$$
Here, $|Du|^{p(x)-p^-}$ and $1+a(x)|Du|^{q-p}$ blow up or vanish 
when $|Du|$ does. Therefore, by identifying $a(x)$ in I with $|Du|^{p(x)-p^-}$ or $1+a(x)|Du|^{q-p}$, we see that $a$ is neither bounded nor far away from the zero. Let us briefly introduce regularity results for the above types. Let $u$ be a minimizer of the $\phi$-energy 
\eqref{mainfunctional} with $\phi$ being one of I--III. Then the following is known:

For type I, i.e.\ $\phi(x,t)=a(x)\psi(t)$, 
suppose $a$ is continuous with modulus of continuity $\omega_a$. Then 
\begin{equation}\label{(1)alpha}
\begin{split}
\lim_{r\to 0^+}\omega_a(r)=0
\quad&\Longrightarrow\quad 
u\in C^{\alpha}\ \text{ for any }\ \alpha\in(0,1), \\
\omega_a(r)\lesssim r^\beta \ \text{for some }\ \beta>0
\quad&\Longrightarrow\quad 
u\in C^{1,\alpha}\ \text{ for some }\ \alpha\in(0,1),
\end{split}
\end{equation}
see for instance \cite{Min06} and references therein.

For type II, i.e.\ $\phi(x,t)=t^{p(x)}$, 
suppose $p$ is continuous with modulus of continuity $\omega_p$. Then
\begin{equation}\label{(2)alpha}
\begin{split}
\lim_{r\to 0^+}\omega_p(r)\ln \tfrac 1r=0
\quad&\Longrightarrow\quad 
u\in C^{\alpha}\ \text{ for any }\ \alpha\in(0,1), \\
\omega_p(r)\lesssim r^\beta \ \text{for some }\ \beta>0
\quad&\Longrightarrow\quad 
u\in C^{1,\alpha}\ \text{ for some }\ \alpha\in(0,1).
\end{split}
\end{equation}
For these results, we refer to the series of papers of Acerbi, Coscia and Mingione \cite{AM01,AM02,CosM99}, see also \cite{AM05,BO1,Fan07,FanZha99}.

For type III, i.e.\ $\phi(x,t)=t^p + a(x)t^q$,
suppose $a\in C^{0,\beta}$ for some $\beta\in(0,1]$. Then
\begin{equation}\label{(3)<=}
\frac qp \leq 1+\frac{\beta}{n}
\quad\Longrightarrow\quad 
u\in C^{1,\alpha}\ \text{ for some }\ \alpha\in(0,1).
\end{equation}
For this result, we refer to the series of papers of Baroni, Colombo and Mingione \cite{BCM3,ColM15-1}, see also \cite{BCM1,BOh1,ColM15-2,ColM16,Ok1}. Note that no independent condition implies $C^\alpha$-regularity. 
In other words, we cannot ensure even $C^\alpha$-regularity for $u$ if 
$\frac qp>1+\frac{\beta}{n}$. We also mention that the $C^{1,\alpha}$-regularity for type III was first proved under the following condition instead of \eqref{(3)<=}:
\begin{equation}\label{(3)<}
\frac qp < 1+\frac{\beta}{n}
\quad\Longrightarrow\quad 
u\in C^{1,\alpha}\ \text{ for some }\ \alpha\in(0,1),\ \ \ \text{see \cite{ColM15-1}},
\end{equation}
and later it was extended to the borderline case $\frac qp=1+\frac{\beta}n$ in \cite{BCM3}, 
see also \cite{DeFM_pp}.

As mentioned in the introduction, our general results cover all of these special cases. 
Specifically, Theorem~\ref{mainthm1-0}(1) implies \eqref{(1)alpha}$_1$ 
and \eqref{(2)alpha}$_1$ and Theorem~\ref{mainthm1-0}(2) implies \eqref{(1)alpha}$_2$, 
\eqref{(2)alpha}$_2$ and \eqref{(3)<}. We notice that Theorem~\ref{mainthm1-0}(2) 
does not imply \eqref{(3)<=}. In fact, \MA{} holds when $\phi(x,t)=t^p+a(x)t^q$ 
with $a(\cdot)\in C^{0,\beta}$ if and only if the strict inequality $\frac qp<1+\frac{\beta }n$ holds. 
This gap will be filled by Theorem~\ref{mainthm2}; this is one main reason 
why we consider the slightly weaker assumption \wMA{}. 

Furthermore, many other, previously unstudied cases can also be covered, cf., e.g.\ 
Corollary~\ref{cor:radulescu}, and Section~\ref{sect:examples} more generally. 
Originally, the double phase model was introduced to model the situation when 
two phases (the $p$ and the $q$-growth phases) mix. Since only the larger 
exponent affects the nature of the problem, this was simplified in the form 
$t^p + a(x) t^q$ that we have seen. However, we can also consider a variant which is more 
closely related to the original motivation:
\begin{equation}\label{gdoublephase}
\phi(x,t)=(1-a(x)) t^p + a(x) t^q,
\quad\text{where }\ 1<p\leq q,\ \ a(\cdot):\Omega \to [0,1].
\end{equation}
Now $a$ indicates the relative amount of material at a point from the $q$-phase. Such functionals have been treated by Eleuteri--Marcellini--Mascolo \cite{EleMM16a, EleMM16b, EleMM_pp}. 
More generally, we can also deal with general double phase problems of the type
\[
\phi(x,t)=a(x)\psi(t)+b(x)\xi(t),
\]
where $a(\cdot),b(\cdot)\ge 0$ satisfy $\nu\le a(\cdot)+b(\cdot)\le L$ and $\psi',\xi'$ satisfy \azero{}, \inc{p-1} and \dec{q-1}, which includes the following examples:
\[
t^p+a(x)t^q,\quad a(x)t^p+t^q, \quad a(x)t^p+b(x)t^q, 
\quad\text{and}\quad \psi(t)+a(x)\psi(t)\ln(e+t).
\] 
We present conditions for above functions to satisfy \wMA{} or \MA{} 
in Corollaries~\ref{corexdouble1} and \ref{corexdouble2},
so that $C^{\alpha}$- and $C^{1,\alpha}$-regularity results for \eqref{gdoublephase} are obtained as special cases. We note that the second example $a(x)t^p+t^q$ can be understood as a functional with standard $q$-growth and hence $q/p$ has no upper bound to obtain the regularity results. Here, we explain the regularity results for this functional as a special case of double phase problems.
In addition, in the same spirit, one could consider functionals with infinitely many phases such that 
\[
\phi(x,t)= \sum_{i=1}^\infty a_i(x)t^{p_i},\quad \text{where }\ 1<p\le p_i\le q,\ \ a_i(\cdot)\geq 0\ \text{ and }\ 0<\nu\leq \sum_{i=1}^\infty a_i(\cdot)\leq L,
\]
which satisfies the fundamental assumption of  Theorem~\ref{mainthm1-0}.


\section{Generalized Orlicz spaces}\label{sect:preliminary}


\subsection*{Notation and assumptions}
For $x_0\in \Rn$ and $r>0$, $B_r(x_0)$ is the ball in $\Rn$ with radius $r$ and 
center $x_0$. We write $B_r=B_r(x_0)$ when the center is clear or unimportant. 
For an integrable function $f$ in $U\subset \Rn$, we define $(f)_{U}$ by the average of 
$f$ in $U$ in the integral sense, that is,
$(f)_{U} := \fint_U f\, dx := \frac{1}{|U|}\int_U f \,dx$. 
We say that $f:[0,\infty)\to [0,\infty)$ is \textit{almost increasing} or 
\textit{almost decreasing} if there exists $L\ge 1$ such that for any $0<t<s<\infty$, $f(t)\leq Lf(s)$ or $f(s)\leq Lf(t)$, respectively. In particular, if $L=1$ we say $f$ is \textit{non-decreasing} or \textit{non-increasing}.

We refer to \cite{HarH_book} for more details about basics of $\Phi$-functions and 
generalized Orlicz spaces. 
For $\phi:\Omega\times[0,\infty)\to[0,\infty)$ and $B_r\subset \Omega$, we write
$$
\phi^+_{B_r}(t):=\sup_{x\in B_r}\phi(x,t)\quad \text{and}\quad\phi^-_{B_r}(t):=\inf_{x\in B_r}\phi(x,t).
$$
If the map $t\mapsto\phi(x,t)$ is non-decreasing for every $x\in\Omega$, then the (left-continuous) 
inverse function with respect to $t$ is defined by
$$
\phi^{-1}(x,t):= \inf\{\tau\geq 0: \phi(x,\tau)\geq t\}.
$$ 
If $\phi$ is strictly increasing and continuous in $t$, then this is just the normal inverse function. 

\begin{definition} 
Let $\phi:\Omega\times[0,\infty)\to[0,\infty)$ and $\gamma>0$. 
We define some conditions related to regularity with 
respect to the $t$-variable. 
\vspace{0.2cm}
\begin{itemize}
\item[\normalfont(aInc)$_\gamma$]\label{ainc} 
The map $t\mapsto \phi(x,t)/t^\gamma$ is almost increasing with constant $L\geq 1$ uniformly in $x\in\Omega$. 
\vspace{0.2cm}
\item[\normalfont(Inc)$_\gamma$]\label{inc} 
The map $t\mapsto \phi(x,t)/t^\gamma$ is non-decreasing for every $x\in\Omega$. 
\vspace{0.2cm}
\item[\normalfont(aDec)$_\gamma$]\label{adec} 
The map $t\mapsto \phi(x,t)/t^\gamma$ is almost decreasing with constant $L\geq 1$ uniformly in $x\in\Omega$.
\vspace{0.2cm}
\item[\normalfont(Dec)$_\gamma$]\label{dec} 
The map $t\mapsto \phi(x,t)/t^\gamma$ is non-increasing for every $x\in\Omega$.
\vspace{0.2cm}
\item[\normalfont(A0)] \label{azero} There exists $L\geq 1$ such that $L^{-1}\leq \phi(x,1)\leq L$ for every $x\in \Omega$.
\end{itemize}
\end{definition}

Note that this version of \azero{} is slightly stronger than the one used in \cite{HarH_book}, 
but they are equivalent under the doubling assumption \adec{}. 
Let $0<c\leq 1\leq C<\infty$. If $\phi$ satisfies \ainc{\gamma} with 
constant $L\geq1$, then 
\[
\phi(x,ct)\leq Lc^\gamma\phi(x,t)\quad\text{and}\quad L^{-1}C^\gamma \phi(x,t)\leq \phi(x,Ct) \quad\text{for all }\ (x,t)\in \Omega\times[0,\infty).
\]
On the other hand, if $\phi$ satisfies \adec{\gamma} with the constant $L\geq 1$, then 
\[
L^{-1} c^\gamma \phi(x,t)\leq \phi( x,ct) \quad\text{and}\quad \phi(x,Ct)\leq LC^\gamma\phi(x,t) \quad\text{for all }\ (x,t)\in\Omega\times[0,\infty).
\]
\begin{remark}
If $\phi$ satisfies \ainc{\gamma} or \adec{\gamma} for some $\gamma>0$, then so do $\phi_{B_r}^-$ and $\phi_{B_r}^+$ for any $B_r\subset\Omega$.
\end{remark}

\begin{remark}\label{rmkincdec}
Suppose that $\phi(x,\cdot)\in C^1([0,\infty))$ for each $x\in\Omega$ and that $\gamma>0$. 
Then 
\begin{itemize}
\item 
$\phi$ satisfies \inc{\gamma} if and only if 
$\gamma \phi(x,t)\leq t\phi'(x,t)$ for all $x\in\Omega$ and $t\in[0,\infty)$;
\item
$\phi$ satisfies \dec{\gamma} if and only if 
$\gamma \phi(x,t)\geq t\phi'(x,t)$ for all $x\in\Omega$ and $t\in[0,\infty)$.
\end{itemize}
These conclusions are obtained by differentiating the function $t\mapsto\phi(x,t)/t^\gamma$.
\end{remark}

For functions $f,g:U\to \R$ with $U\subset \Rn$, 
$f\lesssim g$ or $f\approx g$ (in $U$) mean that there exists $C\geq 1$ such that $f(y)\leq C g(y)$ or $C^{-1} f(y)\leq g(y)\leq C f(y)$, respectively, for all $y\in U$. In particular, in this paper we shall use these symbols when the relevant constants $C$ depend only on $n$ and constants from the fundamental conditions \ainc{\gamma}, \adec{\gamma}, \inc{\gamma}, \dec{\gamma} and \azero{}. By following this, for instance, \azero{} can be written as $\phi(\cdot,1)\approx 1$ in $\Omega$. 
We use some results from papers with a weaker notion of 
equivalence: $f\simeq g$ (in $U$) which means that there exists $C\geq 1$ such that $f(C^{-1}y)\leq g(y)\leq f(Cy)$ for all $y\in U$. However, if \adec{} holds, 
then $\simeq$ and $\approx$ are equivalent and furthermore constants can be moved 
inside and outside of $\phi$ as observed above. 

\subsection*{Basic properties of generalized \texorpdfstring{$\phi$}{Phi}-functions and related functions spaces}
We next introduce classes of $\Phi$-functions. Let $L^0(\Omega)$ be the set of the 
measurable functions on $\Omega$. In the sequel we omit the words ``generalized'' and ``weak''  
from the parentheses. 

\begin{definition}\label{defPhi}
Let $\phi:\Omega\times[0,\infty)\to [0,\infty]$. We call 
$\phi$ a \textit{(generalized) $\Phi$-prefunction} if $x\mapsto \phi(x,|f(x)|)$ is measurable for every $f\in L^0(\Omega)$, and $t\mapsto \phi(x,t)$ is non-decreasing for 
every $x\in\Omega$ and satisfies that 
$\phi(x,0)=\lim_{t\to0^+}\phi(x,t)=0$ and $\lim_{t\to\infty}\phi(x,t)=\infty$
for every $x\in\Omega$. A prefunction $\phi$ is a 
\begin{itemize}
\item[(1)] \textit{(generalized weak) $\Phi$-function}, denoted $\phi\in\Phiw(\Omega)$, if it satisfies \ainc{1}.
\item[(2)] \textit{(generalized) convex $\Phi$-function}, denoted $\phi\in\Phic(\Omega)$, if $t\mapsto \phi(x,t)$ is left-continuous and convex for every $x\in\Omega$.
\end{itemize}
If $\phi$ is independent of $x$, then we denote $\phi\in\Phiw$ 
or $\phi\in\Phic$ without ``$(\Omega)$''. 
\end{definition}

We note that convexity implies \inc{1} so that
$\Phic(\Omega)\subset \Phiw(\Omega)$.
For $\phi\in\Phiw(\Omega)$, the \textit{generalized Orlicz space} (also known as the \textit{Musielak--Orlicz space}) is defined by 
\[
L^{\phi}(\Omega):=\big\{f\in L^0(\Omega):\|f\|_{L^\phi(\Omega)}<\infty\big\}
\] 
with the (Luxemburg) norm 
\[
\|f\|_{L^\phi(\Omega)}:=\inf\bigg\{\lambda >0: \varrho_{\phi}\Big(\frac{f}{\lambda}\Big)\leq 1\bigg\},
\ \ \text{where}\ \ \varrho_{\phi}(f):=\int_\Omega\phi(x,|f(x)|)\,dx.
\]
We denote by $W^{1,\phi}(\Omega)$ the set of $f\in L^{\phi}(\Omega)$ satisfying that $\partial_1f,\dots,\partial_nf \in L^{\phi}(\Omega)$, where $\partial_if$ is the weak derivative of $f$ in the $x_i$-direction, with the norm $\|f\|_{W^{1,\phi}(\Omega)}:=\|f\|_{L^\phi(\Omega)}+\sum_i\|\partial_if\|_{L^\phi(\Omega)}$. Note that if $\phi$ satisfies \adec{q} for some $q\ge 1$, then $f\in L^\phi(\Omega)$ if and only if $\varrho_\phi(f)<\infty$, and if $\phi$ satisfies \azero{}, \ainc{p} and \adec{q} for some $1<p\leq q$, then $L^\phi(\Omega)$ and $W^{1,\phi}(\Omega)$ are reflexive Banach spaces. In addition we denote by $W^{1,\phi}_0(\Omega)$ the closure of $C^\infty_0(\Omega)$ in $W^{1,\phi}(\Omega)$. For more information about the generalized Orlicz and Orlicz--Sobolev spaces, we refer to the monographs 
\cite{HarH_book,LanM19} and also \cite[Chapter~2]{DieHHR11}.

\smallskip

For $\phi:[0,\infty)\to [0,\infty)$, we define the conjugate function by
$$
\phi^*(x,t) :=\sup_{s\geq 0} \, (st-\phi(x,s)).
$$ 
By definition, we have the following Young inequality: 
$$
ts\leq \phi(x,t)+\phi^*(x, s)\quad \text{for all }\ s,t\ge 0.
$$
If $\phi\in \Phic(\Omega)$, then $(\phi^*)^*=\phi$ \cite[Theorem~2.2.6]{DieHHR11}. 

We state some properties of $\Phi$-functions, for which we refer to 
\cite[Chapter 2]{HarH_book}. 

\begin{proposition}\label{prop00} 
Let $\phi$ be a $\Phi$-prefunction. 
\begin{itemize}
\item[(1)] 
If $\phi$ satisfies \ainc{1}, then there exists $\psi\in\Phic(\Omega)$ such that $\phi\simeq\psi$. 
\item[(2)] 
If $\phi$ satisfies \adec{1}, then there exists $\psi\in\Phic(\Omega)$ such that $\phi\approx\psi^{-1}$. Note that $\psi^{-1}(x,\cdot)$ is concave.
\item[(3)] 
Let $p,q\in (1,\infty)$. Then $\phi$ satisfies \ainc{p} or 
\adec{q} if and only if $\phi^*$ satisfies \adec{\frac{p}{p-1}} or \ainc{\frac{q}{q-1}}, respectively. 
\item[(4)]
Let $\phi\in \Phiw(\Omega)$ and $\gamma\ge 1$. Then 
$\phi$ satisfies \ainc{\gamma} or \adec{\gamma} if and only if $\phi^{-1}$ satisfies \adec{1/\gamma} or \ainc{1/\gamma}, respectively.
\item[(5)] 
If $\phi$ satisfies \ainc{p} and \adec{q}, then for any $s,t\ge 0$ and $\kappa\in(0,1)$,
\[
ts 
\leq 
\phi(x,\kappa^{\frac{1}{p}}t)+\phi^*(x,\kappa^{-\frac{1}{p}}s) 
\lesssim 
\kappa \phi(x,t)+\kappa^{-\frac{1}{p-1}} \phi^*(x,s)
\]
and
\[
ts 
\leq 
\phi(x,\kappa^{-\frac{1}{q'}}t)+\phi^*(x,\kappa^{\frac{1}{q'}}s) 
\lesssim 
\kappa^{-(q-1)} \phi(x,t)+\kappa \phi^*(x,s).
\]
\end{itemize}
\end{proposition}

If $\phi\in\Phic(\Omega)$, then there exists $\phi'=\phi'(x,t)$, which is non-decreasing and right-continuous, satisfying that
$$
\phi(x,t)=\int_0^t \phi'(x,s)\, ds.
$$
Such $\phi'$ is called the right-derivative of $\phi$. Note that this 
derivative was denoted by $\partial_t\phi$ in the introduction. 
We next collect some results about the derivative $\phi'$. 
For (4), we give a simple direct 
proof, since earlier proofs of the inequality used additional assumptions. 

\begin{proposition}\label{prop0} 
Let $\gamma>0$ and suppose that $\phi\in\Phic(\Omega)$ with derivative $\phi'$.
\begin{itemize}
\item[(1)] 
If $\phi'$ satisfies \ainc{\gamma}, \adec{\gamma}, \inc{\gamma} or \dec{\gamma}, then $\phi$ satisfies \ainc{\gamma+1}, \adec{\gamma+1}, \inc{\gamma+1} or \dec{\gamma+1}, respectively, with the same constant $L\geq 1$. 
\item[(2)] 
If $\phi'$ satisfies \adec{\gamma} with constant $L$, 
then $\phi(x,t)\approx t \phi'(x,t)$, more precisely
\[
\frac{t\phi'(x,t)}{2^{\gamma+1}L}\leq \phi(x,t) \leq t\phi'(x,t) \quad \text{for }\ (x,t)\in \Omega\times[0,\infty).
\]
\item[(3)] 
If $\phi'$ satisfies \azero{} and \adec{\gamma} with constant $L\geq 1$, then $\phi$ also satisfies \azero{}, with constant depending on $L$ and $\gamma$.
\item[(4)] 
$\phi^*(x,\phi'(x,t))\le t\phi'(x,t)$. 
\end{itemize}
\end{proposition}
\begin{proof}
We start with (1) and suppose that $\phi'$ satisfies \ainc{\gamma}. 
Fix $0<t< s<\infty$ and set $a:=\frac st> 1$. Then \ainc{\gamma} 
of $\phi'$ implies that
\begin{align*}
\frac{\phi(x,t)}{t^{\gamma+1}}
&=\frac{1}{t^{\gamma+1}} \int_0^{t}\phi'(x,\tau)\, d\tau\\
& \leq \frac{L}{t^{\gamma+1}} \int_0^{t}\frac{\phi'(x,a\tau)}{a^\gamma}\, d\tau \overset{\widetilde\tau=a\tau}{=}\frac{L}{(at)^{\gamma+1}} \int_0^{at}\phi'(x,\widetilde\tau)\, d\tilde \tau= L\frac{\phi(x,s)}{s^{\gamma+1}},
\end{align*}
which means $\phi$ satisfies \ainc{\gamma+1}. In the same way we can also prove that \adec{\gamma} of $\phi'$ implies \adec{\gamma+1} of $\phi$. The claims regarding \inc{} and 
\dec{} follow when $L=1$. 

We next prove (2). Since $\phi'$ is non-decreasing, it follows that 
\[
\tfrac t2\phi'(x,\tfrac t2) \le \underbrace{ \int_0^t \phi'(x,\tau)\, d\tau}_{=\phi(x,t)}\le t\phi'(x,t).
\]
By the \adec{\gamma} condition of $\phi'$, we have 
$\phi'(x,\tfrac t2)\ge L^{-1}2^{-\gamma} \phi'(x,t)$, which implies 
$\phi(x,t)\approx t \phi'(x,t)$.

Then, we prove (3). By (2) and \azero{} of $\phi'$ it follows that 
$\phi(\cdot,1)\approx 1\cdot \phi'(\cdot,1)\approx 1$, so $\phi$ satisfies \azero{}. 

Finally, we prove (4). Since $\phi$ is convex, 
$\phi(x,s)\ge \phi(x,t) + k(s-t)$, where $k:=\phi'(x,t)$ is the slope. 
Then from the definition of the conjugate function we have 
\[
\phi^*(x,\phi'(x,t)) 
=
\sup_{s\geq 0} (sk-\phi(x,s))
\le 
\sup_{s\geq 0} (sk - \phi(x,t) - k(s-t))
= tk - \phi(x,t)
\le 
t\phi'(x,t). \qedhere
\]
\end{proof}

We end this subsection with some properties for $C^1$-regular $\Phi$-functions. 
Note that Proposition~\ref{prop000}(2) below is proved for $C^2$-functions 
in \cite[Lemma~3]{DieE08} -- here we provide a more elementary proof
which is based on a reduction to the same claim for the function $t^p$, that is
\begin{equation}\label{eq:polyEst}
\big(|x|^{p-2} x-|y|^{p-2} y\big) \cdot (x-y)
\approx
(|x|+|y|)^{p-2} |x-y|^2 \quad\text{for } p>1. 
\end{equation}
While versions of this claim are commonly known, we have not found this precise 
formulation in the literature. Rather than providing a proof of \eqref{eq:polyEst}, 
we just invoke \cite[Lemma~3]{DieE08}, since $t^p$ is certainly a $C^2$-function. 

\begin{proposition}\label{prop000}
Let $\phi\in \Phic\cap C^1([0,\infty))$ with $\phi'$ satisfying 
\inc{p-1} and \dec{q-1} for some $1<p\leq q$. 
Then for $\kappa\in(0,\infty)$ and $x,y\in \Rn$ the following hold:
\begin{enumerate}
\item
$\displaystyle \frac{\phi'(|x|+|y|)}{|x|+|y|}|x-y|^2\approx
\Big(\frac{\phi'(|x|)}{|x|}x-\frac{\phi'(|y|)}{|y|}y\Big) \cdot (x-y)$;
\vspace{6pt}
\item
$\displaystyle \frac{\phi'(|x|+|y|)}{|x|+|y|}|x-y|^2 
\lesssim \phi(|x|) -\phi(|y|)-\frac{\phi'(|y|)}{|y|}y\cdot(x-y)$; 
\vspace{6pt}
\item
$\displaystyle \phi(|x-y|) \lesssim 
\kappa\left[\phi(|x|)+\phi(|y|)\right]+ \kappa^{-1}\frac{\phi'(|x|+|y|)}{|x|+|y|}|x-y|^2$.
\end{enumerate}
If additionally $\phi\in C^2((0,\infty))$, then 
$t\phi''(t)\approx\phi'(t)$ and $\frac{\phi'(|x|+|y|)}{|x|+|y|}$ 
can be replaced by $\phi''(|x|+|y|)$. 
\end{proposition} 

\begin{proof}
When $\phi\in C^2((0,\infty))$, the inequalities $t\phi''(t)\approx\phi'(t)$ are direct consequences of 
Remark~\ref{rmkincdec} and Proposition~\ref{prop0}(1). This also implies 
$\frac{\phi'(|x|+|y|)}{|x|+|y|}\approx \phi''(|x|+|y|)$.

For (1), we may assume without loss of generality that $|x|\ge |y|$. By \inc{p-1}
and \dec{q-1}, 
\[
\Big(\frac{|y|}{|x|}\Big)^{q-1} \phi'(|x|)
\le
\phi'(|y|) 
\le 
\Big(\frac{|y|}{|x|}\Big)^{p-1} \phi'(|x|). 
\]
Thus there exists $\gamma\in [p-1,q-1]$ such that 
$\phi'(|y|) = (\frac{|y|}{|x|})^\gamma \phi'(|x|)$. 
Hence
\[
\Big(\frac{\phi'(|x|)}{|x|}x-\frac{\phi'(|y|)}{|y|}y\Big) \cdot (x-y)
=
\frac{\phi'(|x|)}{|x|^\gamma} 
\big(|x|^{\gamma-1}x-|y|^{\gamma-1}y\big) \cdot (x-y).
\]
We use \eqref{eq:polyEst} with $\gamma+1$ in place of $p$.
Furthermore, from $|x|\ge |y|$ it follows that $|x|+|y|\approx 
|x|$, and so we have 
\[
\Big(\frac{\phi'(|x|)}{|x|}x-\frac{\phi'(|y|)}{|y|}y\Big) \cdot (x-y)
\approx 
\frac{\phi'(|x|)}{|x|^\gamma} (|x|+|y|)^{\gamma-1} |x-y|^2
\approx 
\frac{\phi'(|x|+|y|)}{|x|+|y|} |x-y|^2.
\]

We next prove (2). Denote $\eta:= \frac{x-y}{|x-y|}$ and 
$z_s:=y+\eta s$. Then 
\[
\phi(|x|) - \phi(|y|) = \int_0^{|x-y|} \phi'(|z_s|)\frac{z_s}{|z_s|} \cdot \eta\, ds. 
\]
Furthermore, since $x-y=\eta |x-y|$, we have 
\begin{align*}
\phi(|x|) -\phi(|y|)-\frac{\phi'(|y|)}{|y|}y\cdot(x-y)
&=
\fint_0^{|x-y|} \Big( \frac{\phi'(|z_s|)}{|z_s|}z_s - \frac{\phi'(|y|)}{|y|}y\Big)\cdot (x-y)\, ds \\
&\approx 
\fint_0^{|x-y|} \frac{\phi'(|z_s|+|y|)}{|z_s|+|y|} |x-y|\, s\,ds,
\end{align*}
where the second step follows from (1) since $x-y = \frac{|x-y|} s (z_s-y)$. 
When $s\ge \frac34 |x-y|$, 
$|z_s|+|x|\approx |x|+|y|$ and (2) follows. 

We finally prove (3). By Young's inequality $ab\leq \tfrac{1}{2}(a^2+b^2)$, 
we find that 
\[
|x-y|
\le 
\tfrac12 \kappa (|x|+|y|) + \tfrac12 \kappa^{-1} (|x|+|y|)^{-1} |x-y|^2
\]
Therefore, since $\phi'$ is non-decreasing and $|x-y|\le |x|+|y|$, we find by $t\phi'(t)\approx \phi(t)$ that 
\begin{align*}
\phi(|x-y|) 
&\lesssim 
\phi'(|x|+|y|)|x-y|\\
&\le 
\kappa \phi'(|x|+|y|) (|x|+|y|) 
+ \kappa^{-1} \phi'(|x|+|y|)(|x|+|y|)^{-1} |x-y|^2 \\
&\approx 
\kappa [\phi(|x|)+\phi(|y|)] +\kappa^{-1}\frac{\phi'(|x|+|y|)}{|x|+|y|}|x-y|^2. 
\qedhere
\end{align*} 
\end{proof}


\section{Preliminary regularity results}\label{sect:preregularity}

\subsection*{Assumptions for higher regularity} 

Here we introduce the new assumptions that are used to obtain $C^\alpha$-regularity for any $\alpha\in(0,1)$ or $C^{1,\alpha}$-regularity for some $\alpha\in(0,1)$ of local minimizers of \eqref{mainfunctional}.
We also restate the definition of \MA{} from the introduction, so that 
it can be more easily compared with its weaker variant, \wMA{}. 

In the next definition, we have several conditions which are assumed to hold 
``for any small ball''; this means that it holds for all $r<r_0$ for some $r_0>0$. 

\begin{definition} \label{Def41}
Let $\phi\in \Phiw(\Omega)$. We define some conditions related to regularity with 
respect to the $x$-variable. 
\vspace{0.2cm}
\begin{itemize}
\vspace{0.2cm}
\item[\normalfont(A1)] \label{aone} 
There exists $L\geq 1$ such that for any $B_r\Subset \Omega$ with $|B_r|<1$,
\[
\phi_{B_r}^+(t)\le L \phi_{B_r}^-(t) \quad\text{for all }\ t>0\ \ 
\text{with }\ \phi_{B_r}^-(t)\in [1,|B_r|^{-1}].
\]
\item[\normalfont{(VA1)}] \label{VA1}
There exists a non-decreasing continuous function $\omega:[0,\infty)\to[0,1]$ with $\omega(0)=0$ 
such that for any small $B_r\Subset \Omega$,
\[
\phi^+_{B_r}(t)\leq (1+\omega(r))\phi^-_{B_r}(t)\quad\text{for all }\ t>0\ \ 
\text{with }\ \phi^-_{B_r}(t)\in[\omega(r),|B_r|^{-1}].
\]
\item[\normalfont{(wVA1)}] \label{wVA1} 
For any $\epsilon>0$, there exists a 
non-decreasing continuous function $\omega=\omega_\epsilon:[0,\infty)\to[0,1]$ with 
$\omega(0)=0$ such that for any small ball $B_r\Subset \Omega$,
\[
\phi^+_{B_r}(t)\leq (1+\omega(r))\phi^-_{B_r}(t)
+ \omega(r) \quad\text{for all }\ t>0\ \ 
\text{with }\ 
\phi^-_{B_r}(t)\in[\omega(r),|B_r|^{-1+\epsilon}] .
\]
\end{itemize}
\end{definition}

Intuitively, \aone{} is a jump-condition that restricts the amount that 
$\phi$ can jump between nearby points, whereas \MA{} and \wMA{} are continuity 
conditions that imply continuity with respect to the $x$-variable.  



\begin{remark}\label{rmkMA1}
We see that \MA{} implies \wMA{} which in turn implies \aone{}. 
Assumption \MA{} is easier to understand but we emphasize that \wMA{} covers an 
interesting borderline case which has arisen in the double phase case, cf.\ 
Corollary~\ref{corexdouble2}. 
\end{remark}

\begin{remark}\label{rem:pdeCase}
Finally, we would like to explain why we adapt the methodology of calculus of variations, instead of one of partial differential equations, since indeed 
$u$ is a minimizer of \eqref{mainfunctional} if and only if it is a weak solution to
\[ 
\div \left(\frac{\phi'(x,|Du|)}{|Du|}Du\right)=0\quad \text{in }\ \Omega,
\] 
see \cite{HarHK16}.
In the comparison step in our approach, we take advantage of the minimizing property of $u$. If we would instead use the PDE approach, to the best of our understanding, the main assumption \MA{} would be replaced by
the assumption
\[ 
(\phi')^+_{B_r}(t)\leq (1+\omega(r))(\phi')^-_{B_r}(t)\quad\text{for all }\ t>0\ \ \text{satisfying }\ \phi^-_{B_r}(t)\in[\omega(r),|B_r|^{-1}].
\] 
Compared with \MA{}, $\phi$ is replaced by $\phi'$ in the inequality. Since small values are not covered in this assumption or \MA{}, these two 
assumptions are not comparable, i.e.\ one may hold but not the other, in either 
direction. However, if $\phi$ satisfies the basic assumption in Theorem~\ref{mainthm1-0} (this is always assumed in our main theorems), we show that \wMA{} is implied by this assumption: for any $\epsilon>0$, any small $B_r\Subset\Omega$, 
any $t>0$ satisfying $\phi^-_{B_r}(t)\in[\omega(r),|B_r|^{-1+\epsilon}]\subset [\omega(r),|B_r|^{-1}]$ and any $x,y\in B_r$,
\begin{align*}
\phi(x,t)  = 
\int_0^t \phi'(x,s)\, ds & \le 
(1+\omega(r)) \int_{(\phi^-_{B_r})^{-1}(\omega(r))}^t (\phi')^-_{B_r}(s) \, ds  + 
\int_0^{(\phi^-_{B_r})^{-1}(\omega(r))} \phi'(x,s) \, ds\\
&\le 
(1+\omega(r))\int_0^t \phi'(y,s) \, ds + 
\phi(x,(\phi^-_{B_r})^{-1}(\omega(r)))\\
&\le (1+\omega(r))\phi(y,t) + c \omega(r)^{\frac{p}{q}}.
\end{align*}
Thus \wMA{} holds with function $c\omega(r)^{p/q}$. Furthermore, we could also consider a 
\wMA{}-type assumption with $\phi'$ instead of $\phi$, but the same 
argument shows that this also implies \wMA{}.

We note that such difference between regularity assumptions for the minimizer and the PDE problem does 
not appear in types I--III. This also shows that regularity theory for general 
$\phi(x,t)$ cannot be understood easily by just mixing the ones for types I--III.
\end{remark}


\subsection*{Higher integrability and reverse H\"older type inequality}
We prove higher integrability of minimizers of \eqref {mainfunctional} and, as a 
corollary, a reverse H\"older type inequality. In this subsection we 
assume \aone{}.

The following higher integrability result appears as \cite[Theorem~1.1]{HarHK18} in the case 
$\delta=1$. From the proof in that article, one can derive the stated dependence on $\delta$ with the 
help of the \adec{q} assumption; alternatively, one can use that result and a covering argument.

\begin{lemma}[Higher integrability]\label{lemhigh} 
Let $\phi\in\Phiw(\Omega)$ satisfy \azero, \aone{}, \ainc{p} and \adec{q} with 
constant $L\geq 1$ and $1<p\leq q$. 
If $u\in W^{1,\phi}_{\loc}(\Omega)$ is a local minimizer of \eqref {mainfunctional}, 
then there exists $\sigma_0=\sigma_0(n,p,q,L)> 0$, $c_1=c_1(n,p,q,L)\geq 1$ and 
$\sigma_1=\sigma_1(\sigma_0,n,q)$ such that 
\begin{equation}\label{high0}
\bigg(\fint_{B_r} \phi(x,|Du|)^{1+\sigma_0}\,dx\bigg)^{\frac{1}{1+\sigma_0}} 
\leq 
c_1 \delta^{-\sigma_1} \bigg(\fint_{B_{(1+\delta)r}} \phi(x,|Du|)\,dx +1 \bigg)
\end{equation}
for any $B_{2r}\Subset \Omega$ with $\| D u \|_{L^\phi(B_{2r})}\le 1$ and $\delta\in(0,1]$.
\end{lemma}

\begin{remark}\label{rmkRomega'}
Fix $\Omega'\Subset \Omega$. Since $\int_{\Omega'}\phi(x,|Du|)\,dx <\infty$, there exists $R>0$ such that 
$$
\int_{B_r}\phi(x,|Du|)\,dx \leq 1\quad\left(\text{or, equivalently,}\quad \|Du\|_{L^{\phi}(B_r)}\leq 1\right)
$$
for $B_r\subset\Omega'$ with $r\leq R$. 
In view of the previous lemma, this means that $\phi(\cdot,|Du|)\in L^{1+\sigma_0}_{\loc}(\Omega)$. 
\end{remark}

The next lemma contains reverse H\"older type estimates for $Du$. 

\begin{lemma}\label{reverse} 
Let $\phi\in\Phiw(\Omega)$ satisfy \azero, \aone{}, \ainc{p} and \adec{q} with 
constant $L\geq 1$ and $1<p\leq q$. 
Suppose that $u\in W^{1,\phi}_{\loc}(\Omega)$ is a local minimizer of \eqref {mainfunctional} 
and $B_{2r}\Subset \Omega$ with $\| D u \|_{L^\phi(B_{2r})}\le 1$.
There exist $\sigma_0=\sigma_0(n,p,q,L)$ and, for every $t\in(0,1]$, $c_t=c(n,p,q,L,t)>0$ such that 
\begin{equation}\label{high1}
\bigg(\fint_{B_r} \phi(x,|Du|)^{1+\sigma_0}\,dx\bigg)^{\frac{1}{1+\sigma_0}} 
\leq 
c_t\left(\bigg(\fint_{B_{2r}} \phi(x,|Du|)^t\,dx\bigg)^{\frac{1}{t}} +1 \right)
\end{equation}
and $c=c(n,p,q,L)\geq 1$ such that
\[
\fint_{B_r} \phi(x,|Du|)\,dx\leq \left(\fint_{B_r} \phi(x,|Du|)^{1+\sigma_0}\,dx\right)^{\frac{1}{1+\sigma_0}} \leq c\left(\phi_{B_{2r}}^-\left(\fint_{B_{2r}} |Du|\,dx\right) +1 \right).
\]
\end{lemma}
\begin{proof}
We start with the first inequality. 
In \eqref{high0} we split $\phi = \phi^\theta \phi^{1-\theta}$ with $\theta\in(0,1)$ and 
use H\"older's inequality with exponents $\frac{1+\sigma_0}\theta$ and 
$\frac{1+\sigma_0}{1+\sigma_0-\theta}$ and Young's inequality with exponents 
$\frac1\theta$ and $\frac1{1-\theta}$:
\begin{align*}
&\bigg(\fint_{B_r} \phi(x,|Du|)^{1+\sigma_0}\,dx\bigg)^{\frac{1}{1+\sigma_0}} \\
&\hspace{2cm}\leq c_1\left[\delta^{-\sigma_1}\bigg(\fint_{B_{(1+\delta) r}} \phi(x,|Du|)^{1+\sigma_0}\,dx\bigg)^\frac\theta{1+\sigma_0} 
\bigg(\fint_{B_{2r}} \phi(x,|Du|)^t\,dx\bigg)^\frac{1-\theta}t +1\right]\\
&\hspace{2cm}
\leq \frac12 \bigg(\fint_{B_{(1+\delta) r}} \phi(x,|Du|)^{1+\sigma_0}\,dx\bigg)^\frac1{1+\sigma_0} + 
c_2 \delta^{-\frac{\sigma_1}{1-\theta}} \bigg(\fint_{B_{2r}} \phi(x,|Du|)^t\,dx\bigg)^\frac{1}t +c_1
\end{align*}
where we denoted $t:=\frac{(1+\sigma_0)(1-\theta)}{1+\sigma_0-\theta}$. 
Now we see from a standard iteration lemma, e.g.\ \cite[Lemma~4.2]{HarHT17}, 
that the first claim holds.  

We move on the the second claim. The first inequality directly follows from H\"older's inequality, hence we prove the second inequality. Taking $t=\frac1q$ in \eqref{high1}, we see that
\[
\bigg(\fint_{B_r} \phi(x,|Du|)^{1+\sigma_0}\,dx\bigg)^{\frac{1}{1+\sigma_0}} 
\leq c_{1/q}
\left[\bigg(\fint_{B_{2r}} \phi_{B_{2r}}^+(|Du|)^\frac{1}{q}\,dx\bigg)^{q} +1\right].
\]
We notice that the map $t\mapsto [\phi^+_{B_{2r}}(t)]^{\frac{1}{q}}$ satisfies \adec{1}, 
since $\phi^+_{B_{2r}}$ satisfies \adec{q}. Therefore, by Jensen's inequality with Proposition~\ref{prop00}(2), we have 
\begin{equation}\label{lem31pf2}
\bigg(\fint_{B_r} \phi(x,|Du|)^{1+\sigma_0}\,dx\bigg)^{\frac{1}{1+\sigma_0}}
\leq c\left[\phi^+_{B_{2r}}\bigg(\fint_{B_{2r}} |Du|\,dx\bigg)+1\right]
\end{equation}
for some $c=c(c_1,q,L)\ge 1$. In addition, since 
$$
\int_{B_{2r}}\phi_{B_{2r}}^-(|Du|)\, dx 
\le 
\int_{B_{2r}}\phi(x,|Du|)\, dx \leq 1 
\quad\left(\Leftrightarrow\quad 
\|Du\|_{L^{\phi}(B_{2r})}\leq 1\right),
$$
it follows by Jensen's inequality that 
$\phi_{B_{2r}}^-(\fint_{B_{2r}}|Du|\, dx) \lesssim |B_{2r}|^{-1}$. If also the inequality
$\phi_{B_{2r}}^-(\fint_{B_{2r}}|Du|\, dx) \ge 1$ holds, then \aone{} implies that 
\[
\phi_{B_{2r}}^+ \left(\fint_{B_{2r}} |Du|\,dx\right)
\leq L \phi_{B_{2r}}^- \left(\fint_{B_{2r}} |Du|\,dx\right),
\]
whereas in the case $\phi_{B_{2r}}^-(\fint_{B_{2r}}|Du|\, dx) \le 1$, \azero{} 
gives an upper bound of $c$ for the right-hand side of \eqref{lem31pf2}. 
\end{proof}


\subsection*{Regularity results for the autonomous case}
In this subsection, we consider $\phi\in \Phic\cap C^1([0,\infty))\cap C^2((0,\infty))$ 
with $\phi'$ satisfying \inc{p-1} and \dec{q-1} for some $1<p\leq q$. 
Fix $v_0\in W^{1,\phi}(B_r)$ and let $v\in v_0+W_0^{1,\phi}(B_r)$ be a solution of the minimization problem
\begin{equation}\label{secondmin}
\min_{w\in v_0+W_0^{1,\phi}(B_r)}\int_{B_r} \phi(|Dw|)\,dx,
\end{equation}
or equivalently a weak solution to 
\begin{equation}\label{secondeq}
\begin{cases}
\mathrm{div}\left(
\frac{\phi'(|Dv|)}{|Dv|}Dv\right)=0 &\quad\text{in}\ \ B_{r}, \\
v=v_0 &\quad\text{on}\ \ \partial B_{r}.
\end{cases}
\end{equation}
We start with the $C^{1,\alpha}$-regularity in the autonomous case, with appropriate 
estimates.

\begin{lemma}\label{lemHolder}
Let $\phi\in \Phic\cap C^1([0,\infty))\cap C^2((0,\infty))$ with 
$\phi'$ satisfying \inc{p-1} and \dec{q-1} for some $1<p\leq q$. 
If $v\in W^{1,\phi}(B_r)$ is a minimizer of \eqref{secondmin} or a weak solution to \eqref{secondeq}, then $Dv\in C^{\alpha_0}_{\loc}(B_r,\Rn)$ for some $\alpha_0\in (0,1)$ with the following estimates: for any $B_\rho(x_0)\subset B_r$,
\begin{equation}\label{Lip}
\sup_{B_{\rho/2}(x_0)} |Dv|\leq c\fint_{B_\rho(x_0)} |Dv|\, dx
\end{equation} 
and, for any $\tau\in(0,1)$,
\begin{equation}\label{C1alpha0}
\fint_{B_{\tau \rho}(x_0)}|Dv-(Dv)_{B_{\tau\rho}(x_0)}|\,dx 
\leq 
c \tau^{\alpha_0}\fint_{B_\rho(x_0)}|Dv|\,dx.
\end{equation}
Here $\alpha_0\in(0,1)$ and $c>0$ depend only on $n$, $p$ and $q$.
\end{lemma}

The previous lemma is expected from \cite{Lie1}. In particular, we refer to \cite{Ba1} for the case $p\geq 2$. However, we cannot find any result treating the case $p<2$ with 
the above estimates in the literature. Hence, we give a proof of the above lemma in Appendix~\ref{app:a}. We also note that \inc{p-1} and \dec{q-1} of $\phi'$ are equivalent to $t\phi''(t)\approx \phi'(t)$ by Remark \ref{rmkincdec}, since we assume $\phi\in C^2((0,\infty))$.

We next state Calder\'on--Zygmund type estimates in $B_r$ with non-zero boundary data.

\begin{lemma}[Calder\'on--Zygmund estimates]\label{lemCZ}
Let $\phi\in \Phic\cap C^1([0,\infty))\cap C^2((0,\infty))$ with $\phi'$ satisfying 
\inc{p-1} and \dec{q-1} for some $1<p\leq q$, and $|B_r|\leq 1$. 
If $v\in W^{1,\phi}(B_r)$ is the minimizer 
of \eqref{secondmin} or the weak solution to \eqref{secondeq}, then there exists 
$c=c(n,p,q,p_1,q_1,L)>0$ such that 
\begin{equation}\label{ex1}
\|\phi(|Dv|)\|_{L^\theta(B_r)} \leq c\, \|\phi(|Dv_0|)\|_{L^\theta(B_r)}
\end{equation}
for any $\theta\in \Phiw(B_r)$ satisfying 
\azero{}, \aone{}, \ainc{p_1} and \adec{q_1} with constant $L\ge 1$ and $1<p_1\leq q_1$.

Moreover, fix $\kappa>0$ and assume that $\int_{B_r}\theta(x,\phi(|Dv_0|))\,dx\le \kappa$.
Then there exists $c=c(n,p,q,p_1,q_1,L)>0$ such that 
\begin{equation}\label{meanCZestimate}
\fint_{B_r}\theta(x,\phi(|Dv|))\,dx 
\leq 
c\big(\kappa^{\frac{q_1}{p_1}-1}+1\big)
\bigg(\fint_{B_r}\theta(x,\phi(|Dv_0|))\,dx + 1\bigg). 
\end{equation}
\end{lemma}
\begin{proof}
In view of known results about gradient estimates for equations of $p$-Laplacian type or \eqref{secondeq}, see for instance \cite{BC1,BR1,MP1},
it is expected that for any $1<s<\infty$ and any Muckenhoupt weight $w\in A_s$,
\begin{equation}\label{weightedestimate}
\int_{B_r} \phi(|Dv|)^{s}w(x)\, dx \leq c \int_{B_r} \phi(|Dv_0|)^sw(x)\, dx,
\end{equation}
where $c>0$ depends only on $n,p,q,s$ and $[w]_{A_s}$ (see Appendix~\ref{app:b} 
for the definition of the Muckenhoupt class $A_s$). 
We outline the proof of \eqref{weightedestimate} in Appendix~\ref{app:b}.

We may assume that $\|\phi(|Dv_0|)\|_{L^\theta(B_r)} < \infty$, since otherwise 
\eqref{ex1} is trivial. Then $\|\phi(|Dv_0|)\|_{L^{p_1}(B_r)} < \infty$ by \ainc{p_1} of $\theta$
and so $\|\phi(|Dv|)\|_{L^{p_1}(B_r)} < \infty$ by \eqref{weightedestimate} with $s=p_1$. 
We define $\theta_j(x,t):=\min\{ \theta(x,t), j t^{p_1}\}$, $j>0$, and conclude that 
$\|\phi(|Dv|)\|_{L^{\theta_j}(B_r)} < \infty$.
Since $\phi(|Dv|) \in L^{\theta_j}(B_r)$, extrapolation for the generalized Orlicz functions, 
see \cite[Corollary 5.3.4]{HarH_book}, gives 
\[
\|\phi(|Dv|)\|_{L^{\theta_j}(B_r)}  
\lesssim 
\|\phi(|Dv_0|)\|_{L^{\theta_j}(B_r)} 
\le 
\|\phi(|Dv_0|)\|_{L^{\theta}(B_r)}.
\]
We note that in the 
statement of \cite[Corollary 5.3.4]{HarH_book}, $\phi$ is also assumed to 
satisfy the so-called (A2) condition, which is however not needed if the 
domain $\Omega$ is bounded \cite[Lemma~4.2.3]{HarH_book}, and in our case, $\Omega=B_r$.
Finally, \eqref{ex1} follows from this by monotone convergence: 
$\|\phi(|Dv|)\|_{L^{\theta}(B_r)}  = \lim_{j\to \infty} \|\phi(|Dv|)\|_{L^{\theta_j}(B_r)}$, 
cf.\ \cite[Lemma~3.1.4]{HarH_book}. 

We next prove the second claim, inequality \eqref{meanCZestimate}. 
If $\int_{B_r}\theta(x,\phi(|Dv_0|))\,dx\geq 1$,
then it follows from \eqref{ex1} by \cite[Lemma 3.2.10]{HarH_book} that 
\[
\int_{B_r}\theta(x,\phi(|Dv|))\,dx 
\leq
c\left(\int_{B_r}\theta(x,\phi(|Dv_0|))\,dx\right)^{\frac{q_1}{p_1}}
\leq 
c \kappa^{\frac{q_1}{p_1}-1}\int_{B_r}\theta(x,\phi(|Dv_0|))\,dx,
\]
which implies \eqref{meanCZestimate}. 

Now, we suppose that $\int_{B_r}\theta(x,\phi(|Dv_0|))\,dx< 1$. 
We assume first that the \aone{} inequality holds also in $[0,1]$, i.e.\ that, for some $L_1\geq 1$, 
\begin{equation}\label{exass}
\theta^+_{B_\rho}(t) \le L_1 \theta^-_{B_\rho}(t)
\quad\text{for all }\ t>0\ 
\ \text{satisfying}\ \ \theta_{B_\rho}^-(t)\in[0,|B_\rho|^{-1}],
\end{equation}
whenever $B_{\rho}\subset B_r$. Define $\theta^\pm(t):= \theta^\pm_{B_r}(t)$,
$$
M:= (\theta^-)^{-1}\bigg(\int_{B_r}\theta(x,\phi(|Dv_0|))\,dx\bigg)
\quad\text{and}\quad
\bar \theta(x,t):=\frac{\theta(x,Mt)}{\theta^-(M)}.
$$
Note that $\theta^-(M)\in [0,1]$.
Then $\bar\theta$ also satisfies \ainc{p_1} and \adec{q_1}, with 
the same constants as $\theta$. We next prove that $\bar \theta$ satisfies \azero{}. 
It is clear that $\bar\theta^-(1)= 1$. On the other hand, since 
$\theta^-(M)\in [0,1]$, we see by \eqref{exass} with $B_\rho=B_r$ that 
$\bar\theta^+(1)= \theta^+(M)/\theta^-(M)\leq L_1$. Finally we show that 
$\bar\theta$ satisfies \aone{}. Let $B_{\rho}\subset B_r$ and consider $t>0$ such 
that $\bar{\theta}^-_{B_\rho}(t)\in[1,|B_\rho|^{-1}]$. Then
$\theta^-_{B_\rho}(Mt)=\bar{\theta}^-_{B_\rho}(t) \theta^-(M) \leq |B_\rho|^{-1}$.
Therefore, in view of \eqref{exass}, we have 
$$
\bar{\theta}^+_{B_\rho}(t)
=\frac{\theta^+_{B_\rho}(Mt)}{\theta^-(M)}
\leq \frac{L_1 \theta^-_{B_\rho}(Mt)}{\theta^-(M)}
= L_1 \bar{\theta}^-_{B_\rho}(t)
$$
so that $\bar \theta$ satisfies the \aone{} condition with constant $L_1$. 

Let $m:=\phi^{-1}(M)$ and set 
$$
\bar v:=\frac{v}{m},
\quad\bar v_0:=\frac{v_0}{m},
\quad\text{and}\quad
\bar \phi(t):=\frac{\phi(mt)}{M}.
$$
Then $\bar \phi'(t)= \frac{\phi'(mt)m}{M}$ and $\bar v\in W^{1,\bar\phi}(B_r)$ is the weak solution to
$$
\div \left(
\frac{{\bar \phi}'(|D\bar v|)}{|D\bar v|}D\bar v\right)=0\quad\text{in}\ \ B_{r}\quad\text{with}\quad\bar v=\bar v_0 \ \ \text{on}\ \partial B_r.
$$
Note that $\bar\phi'$ also satisfies \inc{p-1} and \dec{q-1} with the same constant as $\phi'$.
In addition, by the definitions of $\bar \theta$, $\bar \phi$ and $M$, 
\[
\int_{B_r} \bar\theta(x,\bar\phi(|D\bar v_0|))\,dx 
=
\frac1{\theta^-(M)}\int_{B_r} \theta(x,\phi(|Dv_0|))\,dx\leq 1\ \ \Rightarrow\ \ 
\|\bar \phi(|D\bar v_0|)\|_{L^{\bar\theta}(B_r)}\leq 1.
\] 
Therefore, applying \eqref{ex1} to $(\theta,\phi,v,v_0)=(\bar\theta,\bar\phi,\bar v,\bar v_0)$, we have
\[
\|\bar\phi(|D\bar v|))\|_{L^{\bar\theta}(B_r)} \leq c \|\bar \phi(|D\bar v_0|)\|_{L^{\bar\theta}(B_r)}\leq c
\]
for some $c=c(n,p,q,p_1,q_1,L_1)\geq 1$.
Finally, this implies that 
\[
\frac1{\theta^-(M)}\int_{B_r} \theta(x,\phi(|Dv|))\,dx
=
\int_{B_r} \bar\theta(x,\bar\phi(|D \bar v|))\,dx \leq c
\]
for some $c=c(n,p,q,p_1,q_1,L_1)\geq 1$. In view of the definition of $M$, we have
\begin{equation}\label{strongmeanCZestimate}
\int_{B_r}\theta(x,\phi(|Dv|))\,dx 
\leq 
c\int_{B_r}\theta(x,\phi(|Dv_0|))\,dx
\end{equation}
in the case when \eqref{exass} holds. Note that \eqref{strongmeanCZestimate} 
is stronger than \eqref{meanCZestimate}, but requires the stronger assumption \eqref{exass}.

We return to the case that $\theta$ satisfies \aone{} with normal range
and define 
\[
\tilde\theta(x,t):=\theta(x,t)+t^{p_1}.
\]
It is easy to check that $\tilde \theta$ satisfies \azero{}, \ainc{p_1}, \adec{q_1} and 
$\theta(x,t)\leq \tilde\theta(x,t) \lesssim \theta(x,t)+1$. Let us show that $\tilde \theta$
satisfies \eqref{exass}. Fix $B_{\rho}\subset B_r$ and $t\geq 0$ satisfying  
$\tilde \theta^-_{B_\rho}(t)\in [0,|B_{\rho}|^{-1}]$. 
Then 
\[
\theta^-_{B_\rho}(t) 
=
\tilde\theta^-_{B_\rho}(t) - t^{p_1}
\leq  
|B_{\rho}|^{-1}
\]
If $\theta^-_{B_\rho}(t)\ge 1$, then
\aone{} of $\theta$ implies that 
\[
\tilde \theta^+_{B_\rho}(t) 
=\theta^+_{B_\rho}(t)+t^{p_1} 
\leq L\theta^-_{B_\rho}(t)+t^{p_1}
\leq L(\theta^-_{B_\rho}(t)+t^{p_1})
=  L \tilde \theta^-_{B_\rho}(t). 
\]
On the other hand, if $\theta^-_{B_\rho}(t)\le 1$, 
by \azero{}, \ainc{p_1}  and \adec{q_1} of $\theta$, we have $t\lesssim 1$ and then 
$\theta^+_{B_\rho}(t) \approx t^{p_1} \leq  \theta^-_{B_\rho}(t)$. Hence $\tilde \theta$ 
satisfies \eqref{exass}. Finally, since 
$\int_{B_r}\tilde\theta(x,\phi(|Dv_0|))\,dx \leq c( 1 +|B_r|)\leq \tilde c$, applying
the result \eqref{strongmeanCZestimate} for function $\tilde c^{-1}\tilde \theta(x,t)$, we obtain 
\begin{align*}
\int_{B_r}\theta(x,\phi(|Dv|))\,dx &\leq \int_{B_r}\tilde \theta(x,\phi(|Dv|))\,dx \\
&\leq  c\int_{B_r}\tilde \theta(x,\phi(|Dv_0|))\, dx\\
&\leq  c\int_{B_r}[\theta(x,\phi(|Dv_0|))+1]\, dx,
\end{align*}
which completes the proof of \eqref{meanCZestimate}.
\end{proof}

\section{Comparison results without continuity assumption}
\label{sect:comparison}

Assume that $\phi\in \Phic(\Omega)\cap C^{1}([0,\infty))$ satisfies a stronger 
version of \aone{}: there exists $L\geq 1$ and a non-decreasing continuous function 
$\omega:[0,\infty)\to[0,1]$ with $\omega(0)=0$ such that for any small $B_r\Subset \Omega$,
\begin{equation}\label{eq:aones}
\phi_{B_r}^+(t)\le L \phi_{B_r}^-(t) \quad\text{for all }\ t>0\ \ 
\text{with }\ \phi_{B_r}^-(t)\in [\omega(r),|B_r|^{-1}].
\end{equation}
Note that this condition is implied by \wMA{} with $L=3$ and $\omega=\omega_\epsilon$ for any fixed $\epsilon$.
Further, we assume that $\phi'$ satisfies \azero{} with the same constant $L\geq1$, 
as well as \inc{p-1} and \dec{q-1} for some $1<p\leq q$.

We fix $\Omega'\Subset \Omega$ and consider $B_{2r}=B_{2r}(x_0)\subset \Omega'$ with $r>0$ satisfying that
\begin{equation}\label{rcondi}
r\leq \frac{1}{2},
\quad 
\omega(2r)\leq \frac1L,
\quad |B_{2r}|
\leq 
\min\left\{\frac1{2L}, 2^{-\frac{2(1+\sigma_0)}{\sigma_0}} \left(\int_{\Omega'} \phi(x,|Du|)^{1+\sigma_0}\,dx\right)^{-\frac{2+\sigma_0}{\sigma_0}}\right\},
\end{equation}
where $\sigma_0\in(0,1)$ is given in Lemma~\ref{lemhigh}. Note that $\phi(\cdot,|Du|)\in L^{1+\sigma_0}_{\loc}(\Omega)$, see Remark~\ref{rmkRomega'}. Hence we have from H\"older's inequality and \eqref{rcondi} that
\begin{equation}\label{Dusigmale1}
\int_{B_{2r}}\phi(x,|Du|)^{1+\frac{\sigma_0}{2}}\,dx
\le |B_{2r}|\left(\fint_{B_{2r}}\phi(x,|Du|)^{1+\sigma_0}\,dx\right)^{\frac{1+\sigma_0/2}{1+\sigma_0}} 
\le \frac{1}{2},
\end{equation}
so that
\begin{equation}\label{Dusigmale2}
\int_{B_{2r}}\phi(x,|Du|)\,dx\leq \int_{B_{2r}}\phi(x,|Du|)^{1+\frac{\sigma_0}{2}}\,dx+ |B_{2r}|\leq \frac{1}{2}+\frac{1}{2} = 1.
\end{equation}
Therefore, we can take advantage of the results in Lemmas~\ref{lemhigh} and \ref{reverse}.
For convenience, we write $\phi^{\pm}(t):=\phi^{\pm}_{B_{2r}}(t)$.


\subsection*{Construction of a regularized Orlicz function}
We construct a regularized function $\tphi\in C^1([0,\infty))\cap C^2((0,\infty))$ 
with $t\tphi''(t)\approx\tphi'(t) $, which is independent of the $x$ variable 
and sufficiently close to $\phi(x_0,t)$ in a suitable range of $t$. This procedure is quite delicate since we want improved differentiability
and, moreover, want to find $\tphi$ satisfying in particular the assumptions of
Proposition~\ref{prop:thetaOK}, below. The challenge lies in ensuring that 
$\phi(x,\tphi^{-1}(t))$ satisfies \ainc{1} and \adec{\gamma} with some $\gamma>1$ for small and large values of $t$, as 
we only have the comparison property when $t$ is in some range $[t_1,t_2]$. We approach this problem by requiring $p$-growth for small and large values of $t$. 
This is counter-intuitive, because it means that the resulting function 
is neither a lower nor an upper bound of the original function, in contrast 
to estimates used in previous articles. 

We first define
\begin{equation}\label{t1t2}
B:=B_{2r}=B_{2r}(x_0),\quad t_1:= (\phi^-)^{-1}(\omega(2r)) \quad \text{and}\quad t_2:=(\phi^-)^{-1}(|B|^{-1}).
\end{equation}
Note that it follows from $\omega(2r),|B_{2r}|\le L^{-1}$ in \eqref{rcondi} and \azero{} of $\phi$ that $t_1\leq 1\leq t_2$. 
Let 
\begin{equation}\label{phis2}
\psi_{B}(t):=
\begin{cases}
a_1\, (\frac t{t_1})^{p-1}&\text{if} \ \ 0\leq t< t_1,\\
\phi'(x_0,t)&\text{if}\ \ t_1\leq t\leq t_2,\\
a_2\, (\frac t{t_2})^{p-1}&\text{if} \ \ t_2< t<\infty,
\end{cases}
\end{equation}
where the constants $a_1:=\phi'(x_0,t_1)$ and $a_2 :=\phi'(x_0,t_2)$ are chosen 
so that $\psi_B$ is continuous. 
We then define 
\[
\phi_B(t):= \int_0^t \psi_B(s)\, ds.
\]
Note that these functions depend on $B$ via the center point $x_0$ as well as the values 
$t_1$ and $t_2$.

When $t\in [t_1,t_2]$, the coincidence of derivatives implies that 
\[
\phi_B(t)-\phi(x_0,t) = \phi_B(t_1)-\phi(x_0,t_1) = \tfrac1p t_1 \phi'(x_0,t_1)-\phi(x_0,t_1)
\]
and so, using the facts that 
$\frac1q t_1 \phi'(x_0,t_1)\le \phi(x_0,t_1)\le \frac1p t_1 \phi'(x_0,t_1)$ by \inc{p} and \dec{q} as well as \eqref{eq:aones}, we find that 
\begin{equation}\label{phiB-phi}
0\le \phi_B(t)-\phi(x_0,t) \le (\tfrac qp -1) \phi(x_0,t_1) \approx \phi^-(t_1) = \omega(2r) \quad \text{for all }\ t\in[t_1,t_2].
\end{equation}

Fix $\eta \in C^\infty_0(\R)$ with $\eta\ge0$, $\supp \eta \subset (0,1)$ and $\|\eta\|_1=1$.
We define 
\begin{equation}\label{tphiphi}
\tphi(t)
:=
\int_0^\infty \phi_B(t\sigma) \eta_{r}(\sigma-1)\,d\sigma
=
\int_0^\infty \phi_B(s) \eta_{rt}(s-t)\,ds
\quad\text{where}\quad
\eta_r(t):=\tfrac1r \eta(\tfrac tr);
\end{equation}
the second expression is valid for $t>0$.
From the second formula, we see that $\tphi\in C^\infty((0,\infty))$. 

For the next proof, we recall the following elementary inequalities which 
follow by the mean value theorem for $s\mapsto(1+s)^\gamma$ on $[0,1]$: 
for $\gamma>0$ and $0<s\leq 1$,
\begin{equation}\label{Taylor}
1+\min\{1,2^{\gamma-1}\} \gamma s\leq (1+s)^\gamma \leq 1+\max\{1,2^{\gamma-1}\}\gamma s .
\end{equation}
For the functions defined above, we have the following properties. 

\begin{proposition}\label{prop:approxProp}
Let $\tilde\phi$ be from \eqref{tphiphi}. Then
\begin{enumerate}
\item
$\phi_B(t) \leq \tphi(t)\leq (1+cr)\phi_B(t)\leq c \phi_B(t)$ for all $t>0$
with $c>0$ depending only on $q$. Furthermore, 
\[
0\leq \tphi(t)-\phi(x_0,t)\leq c(r\phi^-(t) + \omega(2r)) \leq c\phi^-(t)
\quad \text{for all }\ t\in[t_1,t_2].
\]
\item
$\tphi\in C^1([0,\infty))$ and it satisfies \azero{}, \inc{p} and \dec{q} while  
$\tphi'$ satisfies \azero{}, \inc{p-1} and \dec{q-1}. In particular, 
$\tphi'(t)\approx t\tphi''(t)$ for all $t>0$.
\item
$\tphi(t)\leq c\phi(x,t)$ for all $(x,t)\in B\times [1,\infty)$, and so $\tphi(t) \le c \left(\phi(x,t) +1\right)$ for all $(x,t)\in B\times[0,\infty)$. 
\end{enumerate}
Here, the constants $c> 0$ depend only on $n$, $p$, $q$ and $L$.
\end{proposition}

\begin{proof}
It follows from the construction that $\psi_B$ satisfies \inc{p-1}, \dec{q-1} and \azero{}.
By Proposition~\ref{prop0}, $\phi$ and $\phi_B$ satisfy \inc{p}, \dec{q} and \azero{}. 

(1) We note that $\eta_{r}(\sigma-1)$ is only nonzero when $\sigma-1\in [0, r]$. 
As $\phi_B$ is increasing and $\eta \geq 0$, we obtain that 
\[
\tphi(t)
=
\int_0^\infty \phi_B(t\sigma) \eta_{r}(\sigma-1)\,d\sigma
\le
\int_1^{1+r} \phi_B((1+r)t) \eta_{r}(\sigma-1)\,d\sigma
= 
\phi_B((1+r)t)
\]
since $\|\eta_{r}\|_1=1$. Similarly, we obtain that 
$\tphi(t) \ge \phi_B(t)$. In addition, 
by \dec{q} of $\phi_B$ and \eqref{Taylor}, we have 
\[
\tphi(t)\leq \phi_B((1+r)t) \leq (1+r)^{q}\phi_B(t) \leq (1+2^{q-1}qr)\phi_B(t)\quad \text{for all } t\ge 0.
\]
By this inequality and \eqref{phiB-phi}, we estimate 
\begin{align*}
\tphi(t)-\phi(x_0,t)
&\le
(1+2^{q-1}qr)\phi_B(t) - \phi_B(t) + c\omega(2r)
=
2^{q-1}qr\phi_B(t) + c\omega(2r) \\
& \lesssim r\phi^-(t) + \omega(2r)
\end{align*}
for all $t\in[t_1,t_2]$, where we also used \eqref{phiB-phi} and \eqref{eq:aones}
to estimate $\phi_B$ in the last step. 
In addition, we know that $\omega(2r)=\phi^-(t_1)\leq \phi^-(t)$ for all $t\in[t_1,t_2]$.
The lower bound follows from $\tphi(t)\ge \phi_B(t)\ge \phi(x_0,t)$.

(2) 
The claims for $\tphi'$ will be derived based on the equality
\begin{equation}\label{tphi'}
\tphi'(t) 
= 
\int_0^\infty \sigma \psi_B(t\sigma) \eta_{r}(\sigma-1)\,d\sigma\quad \text{for }\ t\ge 0
\end{equation}
which is obtained by differentiating under the integral sign. The continuity of $\psi_B$ implies that $\tphi'\in C([0,\infty))$ and so $\tphi\in C^1([0,\infty))$. As in (1), 
since the support of $\eta_r$ is in $[0,r]$, $\|\eta_r\|_1=1$, and since $\psi_B$ is increasing and satisfies \dec{q-1}, we see that 
\[
\psi_B(t)\leq \tphi'(t) \leq (1+r) \psi_B((1+r)t) \leq (1+r)^q \psi_B(t)\leq 2^q \psi_B(t)
\quad\text{for all }\ t>0,
\] 
that is, $\psi_B\approx \tphi'$. Hence we have $\tphi'(1)\approx \psi_B(1)=\phi'(x_0,1)$, which implies that $\tphi'$ satisfies \azero{}. From the expression for $\tphi'$, we also see, since $\psi_B$ satisfies \inc{p-1}, that
\[
\tphi'(\lambda t) 
= \int_0^\infty \sigma \psi_B(\lambda t\sigma) \eta_{r}(\sigma-1)\,d\sigma
\ge \lambda^{p-1}\int_0^\infty \sigma \psi_B(t\sigma) \eta_{r}(\sigma-1)\,d\sigma
= \lambda^{p-1}\tphi'(t)
\]
for $\lambda\geq 1$ and $t>0$. This yields \inc{p-1} of $\tphi'$. 
Similarly we prove that $\tphi'$ satisfies \dec{q-1}. The properties for $\tphi$ follow 
by Proposition~\ref{prop0}.

(3) 
Fix $x\in B$. When $t\in [1,t_2]$, we see by part (1) and \eqref{eq:aones} that 
$\tphi(t)\lesssim \phi^-(t) \le \phi(x,t)$. For $t\geq t_2$, we observe that 
\[
\psi_B(t) 
= 
\phi'(x_0,t_2) \Big(\frac t {t_2}\Big)^{p-1}
\approx
\phi'(x,t_2) \Big(\frac t {t_2}\Big)^{p-1}
\le 
\phi'(x,t),
\]
since $\phi'(x_0,t_2)\approx \phi(x_0,t_2)/t_2 \approx \phi(x,t_2)/t_2 \approx \phi'(x,t_2)$ by \eqref{eq:aones} and $\phi'$ satisfies \inc{p-1}. Then 
\[
\phi_B(t) = 
\phi_B(t_2) + \int_{t_2}^t \psi_B(\tau) \, d\tau 
\lesssim
\phi(x,t_2) + \int_{t_2}^t \phi'(x,\tau) \, d\tau 
=\phi(x,t). \qedhere
\]
\end{proof}

The next result shows the strength of the approach with \ainc{} and \adec{}, 
since it would be difficult to construct an approximating 
function to guarantee \inc{} and \dec{}. 

\begin{proposition}\label{prop:thetaOK} 
For $\tphi$ defined in \eqref{tphiphi} and any $\sigma\in(0,1)$, set 
\[
\theta(x,t):=[\phi(x,\tphi^{-1}(t))]^{1+\sigma}.
\]
Then $\theta\in\Phiw(B_r)$ satisfies \azero{}, \ainc{1+\sigma}, \adec{q(1+\sigma)/p} 
and \aone{}. Here the constants depend only on 
$n$, $p$, $q$ and $L$ (from the assumptions on $\phi$) and are independent of $\sigma$.
\end{proposition}
\begin{proof}
That $\theta\in \Phiw(B_r)$ is clear once we show 
\ainc{1}. As $\phi$ and $\tphi$ satisfy \azero{}, 
so does $\theta$. Now we prove that $\theta$ satisfies \ainc{1+\sigma}, which holds 
if $t\mapsto \phi(x,\tphi^{-1}(t))$ satisfies 
\ainc{1}. Let $t>s>0$, $\tphi(\tau)=t$ and $\tphi(\sigma)=s$. Then 
\[
L\frac{\phi(x,\tphi^{-1}(t))}t \ge \frac{\phi(x,\tphi^{-1}(s))}s
\quad\Longleftrightarrow\quad
L\frac{\phi(x,\tau)}{\tphi(\tau)}\ge \frac{\phi(x,\varsigma)}{\tphi(\varsigma)}.
\]
By Proposition~\ref{prop:approxProp} and Remark~\ref{rmkincdec}, 
we have $\tphi(t)\approx \phi_B(t)$, $\phi(x,t)\approx t\phi'(x,t)$ and $\phi_B(t)\approx t \psi_B(t)$ for all $(x,t)\in B\times(0,\infty)$. Therefore, it suffices to show that $t\mapsto \frac{\phi'(x,t)}{\psi_B(t)}$ is almost increasing. 
Let $t_1$ and $t_2$ be from in \eqref{t1t2}. Then by the definition of $\psi_B$ in \eqref{phis2}
we see that $t\mapsto \frac{\phi'(x,t)}{\psi_B(t)}$ is non-decreasing on 
$(0,t_1]\cup[t_2,\infty)$, since $\phi'$ satisfies \inc{p-1}. 
By \eqref{eq:aones} with the fact that $t\phi'(x,t)\approx \phi(x,t)$, we have 
$\frac{\phi'(x,t)}{\psi_B(t)}\approx1$ in $[t_1,t_2]$. Therefore, we see that $t\mapsto \frac{\phi'(x,t)}{\psi_B(t)}$ is almost increasing. The property \adec{q(1+\sigma)/p} is proved analogously.

Finally, we show that $\theta$ satisfies \aone{}.
Let $B_{\rho}\subset B_r$, and assume that 
$\theta^-_{B_\rho}(t)\in [1,|B_{\rho}|^{-1}]$. 
Then 
\[
\phi^-_{B_\rho}(\tphi^{-1}(t)) 
=
\theta^-_{B_\rho}(t)^{1/(1+\sigma)}
\in [1,|B_{\rho}|^{-1/(1+\sigma)}]
\subset [1,|B_{\rho}|^{-1}].
\]
Therefore, \aone{} of $\phi$ implies that 
\[
\theta^+_{B_\rho}(t) 
=[\phi^+_{B_\rho}(\tphi^{-1}(t))]^{1+\sigma} 
\leq [L\phi^-_{B_\rho}(\tphi^{-1}(t))]^{1+\sigma}
\le L^2 \theta^-_{B_\rho}(t) 
\]
and so $\theta$ satisfies \aone{}.
\end{proof}


\subsection*{Comparison estimates}

Let $\tphi:[0,\infty)\to[0,\infty)$ be the function constructed in the previous subsection. 
We then consider the minimizer $v\in W^{1,\tphi}(B_r)$ of
\begin{equation}\label{functionaltphiv}
\int_{B_r} \tphi(|Dv|)\, dx 
\quad\text{with}\quad v=u \ \ \text{on}\ \partial B_r,
\end{equation}
where $u\in W^{1,\phi}_{\loc}(\Omega)$ is a minimizer of \eqref {mainfunctional}, and 
derive a comparison estimate between the gradients of $u$ and $v$. We note from Proposition~\ref{prop:approxProp}(3) that $u\in W^{1,\tphi}(B_r)$, so it is an appropriate boundary-value function and thus there exists a unique minimizer of \eqref{functionaltphiv}. 
The minimizer $v$ is also a weak solution to
\begin{equation}\label{eqtphiv}
\div \left(
\frac{\tphi'(|Dv|)}{|Dv|}Dv\right)=0\quad\text{in}\ \ B_{r}\quad\text{with}\quad v=u \ \ \text{on}\ \partial B_r.
\end{equation}
Before stating the main comparison result, we observe the following reverse H\"older type estimate for $Du$ and 
Calder\'on--Zygmund type estimate for the problem \eqref{functionaltphiv}.

\begin{lemma}\label{lemreverseDu}
Let $u\in W^{1,\phi}_{\loc}(\Omega)$ be a local minimizer of \eqref{mainfunctional}
and $v\in W^{1,\tphi}(B_r)$ be the minimizer of \eqref{functionaltphiv}, where $B_{2r}\Subset\Omega$ with $r>0$ satisfying \eqref{rcondi}
and $\tphi$ is defined in \eqref{tphiphi}. Then 
\begin{equation}\label{reverseDutphi}
\left(\fint_{B_r}\phi(x,|Du|)^{1+\sigma_0}\,dx\right)^{\frac{1}{1+\sigma_0}} 
\le c\left[
\tphi \left(\fint_{B_{2r}}|Du|\,dx\right)+1\right]
\end{equation}
and
\begin{equation}\label{CZuv}\begin{aligned}
\fint_{B_r}\phi(x,|Dv|)\,dx 
&\leq \left(\fint_{B_r}\phi(x,|Dv|)^{1+\frac{\sigma_0}{2}}\,dx\right)^{\frac{1}{1+\sigma_0/2}}\\
&\le 
c \left(\fint_{B_r}\phi(x,|Du|)^{1+\frac{\sigma_0}{2}}\,dx+1\right)^{\frac{1}{1+\sigma_0/2}}
\end{aligned}\end{equation}
Moreover,
\begin{equation}\label{DvDu}
\fint_{B_r}|Dv|\,dx 
\le c\left(
\fint_{B_{2r}}|Du|\,dx+1\right).
\end{equation}
Here constants $c\geq 1$ depend on $n$, $p$, $q$ and $L$.
\end{lemma}

\begin{proof}
We first prove \eqref{reverseDutphi}. We note that $u$ satisfies the 
reverse H\"older inequality \eqref{high0} for some $c_1=c_1(n,p,q,L)\geq 1$. Then by Lemma~\ref{reverse}, we have
$$
\left(\fint_{B_r}\phi(x,|Du|)^{1+\sigma_0}\,dx\right)^{\frac{1}{1+\sigma_0}} \lesssim \phi^- \left(\fint_{B_{2r}}|Du|\,dx\right)+1\leq \phi(x_0,t_0)+1,
$$
where $t_0:=\fint_{B_{2r}}|Du|\,dx$. This and \azero{} imply that 
\eqref{reverseDutphi} holds when $t_0\leq 1$; we therefore assume that $t_0\geq 1$. 
By Jensen's inequality and \eqref{Dusigmale2},
$$
1\leq t_0\lesssim (\phi^{-})^{-1}\left(\fint_{B_{2r}}\phi^-(|Du|)\,dx\right)
\leq (\phi^{-})^{-1}\left(|B_{2r}|^{-1}\right)= t_2,
$$ 
where $t_2$ is defined in \eqref{t1t2}. Therefore, it follows from 
Proposition~\ref{prop:approxProp}(1) that
$$
\left(\fint_{B_r}\phi(x,|Du|)^{1+\sigma_0}\,dx\right)^{\frac{1}{1+\sigma_0}} 
\lesssim \phi(x_0,t_0) \lesssim \tphi(t_0)= \tphi \left(\fint_{B_{2r}}|Du|\,dx\right).
$$

As for \eqref{CZuv}, we only prove the second inequality, since the first inequality directly follows from H\"older's inequality. Let $\theta\in \Phiw(B_r)$ be from Proposition~\ref{prop:thetaOK} with $\sigma=\sigma_0/2$. By the proposition, we may apply 
Lemma~\ref{lemCZ} with $v_0=u$ and $\kappa=1$ (see 
\eqref{Dusigmale1}) to conclude that 
\begin{align*}
\fint_{B_r}\phi(x,|Dv|)^{1+\frac{\sigma_0}{2}}\,dx& = \fint_{B_r}\theta(x,\tphi(|Dv|))\,dx\\
&\lesssim \fint_{B_r} [\theta(x,\tphi(|Du|))+1]\,dx 
\approx \fint_{B_r}[\phi(x,|Du|) + 1]^{1+\frac{\sigma_0}{2}}\,dx.
\end{align*}

By Jensen's inequality, Proposition~\ref{prop:approxProp}(3), 
\eqref{CZuv} and \eqref{reverseDutphi}
\[
\tphi\bigg(\fint_{B_r}|Dv|\,dx\bigg)
\le\fint_{B_r}\tphi(|Dv|)\,dx
\lesssim\fint_{B_r}[\phi(x,|Dv|) + 1]\,dx
\lesssim\tphi \left(\fint_{B_{2r}}|Du|\,dx\right) +1.
\]
Then \eqref{DvDu} follows when we apply $\tphi^{-1}$ to both sides and use \adec{q} 
to move ``$+1$'' inside $\tphi$.
\end{proof}


\section{Comparison results with continuity assumption}\label{sect:MA}

Assume that $\phi\in \Phic(\Omega)\cap C^{1}([0,\infty))$ satisfies \wMA{}, 
in addition to the assumptions in the beginning of the previous section, i.e.\ \azero{} 
with constant $L\geq1$, as well as \inc{p-1} and \dec{q-1} for some $1<p\leq q$. 
At this stage, we fix 
\begin{equation}\label{epsilon0}
\epsilon_0:= \frac{\sigma_0}{2(2+\sigma_0)}
\end{equation}
where $\sigma_0\in(0,1)$ is determined in Lemma~\ref{lemhigh}. 
We will use \wMA{} for $\epsilon=\epsilon_0$, which fixes $\omega$ in that condition. 
We take $r$ so small that \eqref{rcondi} holds for this $\omega$. 
Now we derive a gradient comparison estimate between $u$ and $v$.

\begin{lemma} \label{lemcom}
Let $u\in W^{1,\phi}_{\loc}(\Omega)$ be a local minimizer of \eqref{mainfunctional}
and $v\in W^{1,\tphi}(B_r)$ be the minimizer of \eqref{functionaltphiv}, where $B_{2r}\Subset\Omega$ with $r>0$ satisfying \eqref{rcondi}
and $\tphi$ is defined in \eqref{tphiphi}. Then 
there exist $\gamma=\gamma(n,p,q,L)\in(0,1)$ and $c=c(n,p,q,L)\geq 1$ such that 
\[
\fint_{B_r}\tphi''(|Du|+|Dv|)|Du-Dv|^2\,dx \leq c\Big(\omega(2r)^\frac{p}{q} + r^{\gamma}\Big)\left(\tphi\left(\fint_{B_{2r}}|Du|\,dx \right)+1\right).
\]
\end{lemma}

\begin{proof} 
By Proposition~\ref{prop:approxProp}(3) and Lemma~\ref{lemreverseDu}, we see that $u\in W^{1,\tphi}(B_r)$ and $v\in W^{1,\phi}(B_r)$. By Proposition~\ref{prop000}(2),
\[
\tphi''(|Du|+|Dv|)|Du-Dv|^2\lesssim \tphi(|Du|)-\tphi(|Dv|) - \frac{\tphi'(|Dv|)}{|Dv|} Dv \cdot (Du-Dv).
\]
Since $u-v\in W^{1,\tphi}_0(B_r)$ and $v$ is a weak solution to \eqref{eqtphiv}, 
\[
\begin{aligned}
&\fint_{B_r}\tphi''(|Du|+|Dv|)|Du-Dv|^2\,dx \\
&\quad\lesssim 
\fint_{B_r}[\tphi(|Du|)- \tphi(|Dv|)]\,dx - 
\underbrace{\fint_{B_r} \tphi'(|Dv|)\frac{Dv}{|Dv|}\cdot(Du-Dv)\,dx}_{=0 \text{ by \eqref{eqtphiv}}}\\
&\quad =
\fint_{B_r}[\tphi(|Du|)- \phi(x,|Du|)] + 
[\phi(x,|Du|)- \phi(x,|Dv|)] + [\phi(x,|Dv|)- \tphi(|Dv|)]\,dx\\
&\quad\leq 
\underbrace{\fint_{B_r}[\tphi(|Du|)- \phi(x,|Du|)]\,dx}_{=:I_1}+ 
\underbrace{\fint_{B_r}[\phi(x,|Dv|)- \tphi(|Dv|)]\,dx}_{=:I_2}\, ;
\end{aligned}
\]
in the last step we used that $\fint_{B_r}[\phi(x,|Du|)- \phi(x,|Dv|)]\,dx\le 0$ since 
$u$ is a minimizer of \eqref{mainfunctional}. 
We shall estimate $I_2$. We split $B_r$ into three regions $E_1$, $E_2$ and $E_3$ 
defined by 
\[
\begin{aligned}
E_1&:=B_r \cap\{\phi^-(|Dv|)\leq \omega(2r)\},\\
E_2&:=B_r \cap\{ \omega(2r) <\phi^-(|Dv|)\leq |B_{2r}|^{-1+\epsilon_0}\},\\
E_3&:=B_r \cap\{|B_{2r}|^{-1+\epsilon_0}<\phi^-(|Dv|)\}.
\end{aligned}
\]
In the set $E_1$, \dec{q} and \azero{} of $\phi$ imply that 
$|Dv|\lesssim \omega(2r)^\frac1{q}$. 
Therefore by \inc{p} and \azero{} of $\phi$ and $\tphi$,
\[
\fint_{B_r}\big|\phi(x,|Dv|)- \tphi(|Dv|)\big| \chi_{E_1}\,dx
\lesssim 
\omega(2r)^\frac p{q} \fint_{B_r} \chi_{E_1}\,dx
\le
\omega(2r)^\frac p{q}. 
\]
In the set $E_3$, Proposition~\ref{prop:approxProp}(3) and the fact that $1<|B_{2r}|^{1-\epsilon_0}\phi^-(|Dv|)$ imply that 
\[\begin{aligned}
\big|\phi(x,|Dv|)- \tphi(|Dv|)\big|
&\lesssim
\phi(x,|Dv|)+1 
\lesssim
\phi(x,|Dv|)\\
&\le
\left[|B_{2r}|^{1-\epsilon_0}\phi^-(|Dv|)\right]^{\frac{\sigma_0}{2}} \phi(x,|Dv|)\\
&\lesssim
r^{\frac{n(1-\epsilon_0)\sigma_0}{2}} \phi(x,|Dv|)^{1+\frac{\sigma_0}{2}}. 
\end{aligned}\]
Integrating this inequality over $E_3$ and using \eqref{epsilon0}, we find that 
\[
\fint_{B_r}\big|\phi(|x,Dv|)- \tphi(|Dv|)\big| \chi_{E_3}\,dx
\lesssim
r^{\frac{n(4+\sigma_0)\sigma_0}{4(2+\sigma_0)}} 
\left(\fint_{B_r}\phi(x,|Dv|)^{1+\frac{\sigma_0}{2}}\,dx\right)^{\frac{\sigma_0}{2+\sigma_0}+\frac{2}{2+\sigma_0}}.
\]
On one hand, by \eqref{CZuv} and \eqref{reverseDutphi}, we have
\[
\left(\fint_{B_r}\phi(x,|Dv|)^{1+\frac{\sigma_0}{2}}\,dx\right)^{\frac{2}{2+\sigma_0}}
\lesssim \tphi\left( \fint_{B_{2r}}|Du|\,dx\right)+1.
\]
On the other hand, by \eqref{Dusigmale1},
\[
\left(\fint_{B_r}\phi(x,|Dv|)^{1+\frac{\sigma_0}{2}}\,dx\right)^{\frac{\sigma_0}{2+\sigma_0}}\le 
 |B_r|^{-\frac{\sigma_0}{2+\sigma_0}} \lesssim r^{-\frac{4n\sigma_0}{4(2+\sigma_0)}}.
\]
Therefore, combining the previous three inequalities, we have 
\[
\fint_{B_r}\big|\phi(|x,Dv|)- \tphi(|Dv|)\big| \chi_{E_3}\,dx
 \lesssim 
r^{\frac{n\sigma_0^2}{4(2+\sigma_0)}}
\left(\tphi\left(\fint_{B_{2r}}|Du|\,dx\right)+1\right).
\]
 
Recall that $t_1$ and $t_2$ are defined in \eqref{t1t2}. 
In the set $E_2$, we observe that 
\[
\omega(2r) <\phi^-(|Dv|)\leq |B_{2r}|^{-1+\epsilon_0}< |B_{2r}|^{-1}
\quad\text{and so}\quad t_1<|Dv|<t_2.
\]
Hence it follows from \wMA{} and Proposition~\ref{prop:approxProp}(1) that 
\begin{align*}
|\phi(x,|Dv|)- \tphi(|Dv|)|
& \le |\phi(x,|Dv|)-\phi(x_0,|Dv|)|+|\phi(x_0,|Dv|) - \tphi(|Dv|)|\\
& \lesssim 
( r+ \omega(r)) \phi^-(|Dv|)+\omega(2r). 
\end{align*}
Therefore, applying \eqref{CZuv} and \eqref{reverseDutphi}, we have 
\[
\fint_{B_r}\big|\phi(|x,Dv|)- \tphi(|Dv|)\big| \chi_{E_2}\,dx 
\lesssim 
(r+\omega(2r)) \left(\tphi\left(\fint_{B_{2r}}|Du|\,dx\right)+1\right).
\]

We have shown that 
\[\begin{aligned}
|I_2|
&\le \fint_{B_r}\big|\phi(x,|Dv|)- \tphi(|Dv|)\big|\,dx\\
&\lesssim 
\left(\omega(2r)^\frac{p}{q} + r^{\frac{n\sigma_0^2}{4(2+\sigma_0)}} +r\right)
\left(\tphi\left(\fint_{B_{2r}}|Du|\,dx\right)+1\right). 
\end{aligned}\]
The estimate for $I_1$ is analogous, with sets $E_i$ defined with 
$Du$ instead of $Dv$.
\end{proof}

The following corollary is the key to the regularity results in the next section. 
Indeed, once we have the estimate from the corollary, 
the main results follow using standard methods. 

\begin{corollary}\label{corcom} 
Under the assumptions of Lemma~\ref{lemcom}, we have
\[
\fint_{B_r}|Du-Dv|\, dx \leq 
c \big(\omega(r)^{\frac{p}{2q^2}}+r^{\gamma_1} \big)\left(\fint_{B_{2r}} |Du|\,dx+1\right)
\]
for some $\gamma_1=\gamma_1(n,p,q,L) \in(0,1)$ and $c=c(n,p,q,L)\geq 1$.
\end{corollary}
\begin{proof}
Set $\widetilde{\omega}(r):=\omega(r)^\frac pq+r^\gamma$ with $\gamma$ 
from Lemma~\ref{lemcom} and note that $\tilde\omega(\cdot)\leq 2$. 
Applying Proposition~\ref{prop000}(3) with 
$\kappa=\widetilde\omega(r)^{\frac{1}{2}}$, Proposition~\ref{prop:approxProp}(3) 
 and Lemmas~\ref{lemcom} and \ref{lemreverseDu}, we find that
\begin{align*}
\fint_{B_r}&\tphi(|Du-Dv|)\,dx \\
&\lesssim \widetilde \omega(r)^\frac{1}{2} \fint_{B_r}[\tphi(|Du|)+\tphi(|Dv|)] \,dx + \widetilde \omega(r)^{-\frac{1}{2}}\fint_{B_r}\tphi''(|Du|+|Dv|)|Du-Dv|^2\,dx \\
&\lesssim 
\widetilde\omega(r)^\frac{1}{2} \fint_{B_r}[\phi(x,|Du|) + \phi(x,|Dv|)+1]\,dx
+ \widetilde\omega(r)^\frac{1}{2}\left(\tphi\left(\fint_{B_{2r}}|Du|\,dx \right)+1\right)\\
&\lesssim \widetilde \omega(r)^\frac{1}{2}\left(\tphi\left(\fint_{B_{2r}}|Du|\,dx \right)+1\right).
\end{align*}
Therefore, by Jensen's inequality and \dec{q} of $\tphi$, we have 
\[
\tphi\left(\fint_{B_r}|Du-Dv|\, dx\right)
\le \fint_{B_r}\tphi(|Du-Dv|)\, dx
\lesssim \tphi\left(\widetilde\omega(r)^{\frac{1}{2q}}\left(\fint_{B_{2r}}|Du|\,dx+1\right)\right).
\]
The claim follows, since $\tphi$ is strictly increasing.
\end{proof}


\section{Proofs of main results}\label{sect:mainresults}

In this section, we prove the main theorems. 
Before starting the proof we introduce a basic iteration lemma. We refer to \cite[Lemma 2.1 in Chapter III]{Gia1}. 
\begin{lemma}\label{lemtech}
Let $f:[0,r_0] \to [0,\infty)$ be a non-decreasing function. Suppose that for 
all $0<\rho\leq r\leq r_0$,
$$
f(\rho)\leq C\left(\left(\frac{\rho}{r}\right)^n+\epsilon\right)f(r)+C r^n
$$ 
with positive constant $C$. Then for any $\tau \in(0,n)$, there exist $\epsilon_1,c>0$ depending only on $n$, $C$ and $\tau$ such that if $\epsilon<\epsilon_1$, then
$$
f(\rho)\leq c\left(\frac{\rho}{r}\right)^{n-\tau}\left(f(r)+r^{n-\tau}\right).
$$ 
\end{lemma}

In the next results, we denote by $\omega$ the function 
from \wMA{} for $\epsilon=\epsilon_0$, cf.\ 
the beginning of Section~\ref{sect:MA}. Likewise, by $L\ge 1$ we denote the 
constant from \azero{}. 
Now let us state and prove our main results.

\begin{theorem}\label{mainthm1}
Let $\phi\in \Phic(\Omega)\cap C^1([0,\infty))$ with $\phi'$ satisfying 
\azero{}, \inc{p-1} and \dec{q-1} for some $1<p\leq q$ 
and let $u\in W^{1,\phi}_{\loc}(\Omega)$ be a local minimizer of \eqref {mainfunctional}.
If $\phi$ satisfies \wMA{}, then $u\in C^{\alpha}_{\loc}(\Omega)$ for any $\alpha\in(0,1)$. 
\end{theorem} 

\begin{proof}
Let $r_0\in (0,1)$ be a sufficiently small positive number which will be determined later. We fix $\Omega'\Subset\Omega$ and assume that \eqref{rcondi} holds $r=r_0>0$. For any $B_{2r}\subset \Omega'$ with $0<2r\leq r_0$, let $v\in W^{1,\tphi}(B_{r})$ be the minimizer of \eqref{functionaltphiv} with $\tphi$ determined in \eqref{tphiphi}.
When $\rho\in (0,\frac r2)$, applying Corollary~\ref{corcom} with 
$\omega_0(r):=\omega(r)^{\frac{p}{2q^2}}+r^{\gamma_1}$, 
\eqref{Lip} and \eqref{DvDu}, we have
\begin{align*}
\int_{B_\rho} |Du|\, dx &\leq \int_{B_{r}} |Du-Dv|\, dx +\int_{B_\rho} |Dv|\,dx\\
&\lesssim \omega_0(2r)\int_{B_{2r}} \left[|Du|+1\right]\, dx +\rho^n\sup_{B_{r/2}} |Dv|\\
&\le \omega_0(r_0)\int_{B_{2r}} \left[|Du|+1\right]\, dx +\rho^n\fint_{B_{r}} |Dv|\, dx\\
&\lesssim
\left( \left(\frac{\rho}{r}\right)^n+\omega_0(r_0)\right) \int_{B_{2r}} |Du|\, dx +r^n.
\end{align*}
Moreover, if $\frac r2\leq \rho\leq 2r$, the above estimate is obvious since then
$\frac 12\leq \frac \rho r$. 

Let $f(\rho):=\int_{B_\rho}|Du|\,dx$, fix $\tau\in (0,n)$ and choose $r_0$ so small that
$$
\omega_0(r_0)=\omega(r_0)^{\frac{1}{2q}}+r_0^{\gamma_1}\leq \epsilon_1,
$$
where $\epsilon_1$ is from Lemma~\ref{lemtech}. 
Then, applying the lemma and using \eqref{Dusigmale2} with $2r=r_0$, 
we have the following Morrey type estimate: 
\begin{equation}\label{morrey}
\int_{B_\rho} |Du|\, dx
\lesssim 
\Big(\frac{\rho}{r_0}\Big)^{n-\tau} \bigg(\int_{B_{r_0}} |Du|\, dx+r_0^{n-\tau}\bigg) \quad \text{for all }\ B_\rho\subset \Omega'\ \text{ with }\ \rho\in(0,r_0]
\end{equation} 
for implicit constant $c=c(n ,p,q,L,\tau)\geq 1$. 
We note that in \eqref{morrey}, $\rho\in(0,r_0]$ and $B_\rho\subset\Omega'$ are arbitrary and the implicit constant is universal.
Therefore, by taking $\tau=1-\alpha$ for each $\alpha\in (0,1)$ in \eqref{morrey}, we have $u\in C^\alpha_{\loc}(\Omega')$ by a Morrey type embedding, see for instance \cite[Chapter 3, Theorem 1.1]{Gia1}. More precisely, we obtain
\[
[u]_{C^\alpha(B_{r_0/2})} \lesssim r^{1-\alpha}_0\fint_{B_{r_0}}|Du|\, dx +1. \qedhere
\]
\end{proof}

Next we prove the second main theorem, $C^{1,\alpha}$-regularity. 

\begin{theorem}\label{mainthm2}
Let $\phi\in \Phic(\Omega)\cap C^1([0,\infty))$ with $\phi'$ satisfying 
\azero{}, \inc{p-1} and \dec{q-1} for some $1<p\leq q$ 
and let $u\in W^{1,\phi}_{\loc}(\Omega)$ be a minimizer of \eqref {mainfunctional}.
If $\phi$ satisfies \wMA{} with 
\[
\omega(r)\lesssim r^{\beta}\quad \text{for all }\ r\in(0,1]\quad \text{and}\quad \text{for some }\ \beta\in(0,1),
\]
then $u\in C^{1,\alpha}_{\loc}(\Omega)$ for some $\alpha \in (0,1)$. Here $\alpha$ depends only on $n$, $p$, $q$, $L$ and $\beta$.
\end{theorem} 
\begin{proof}
Fix $\Omega'\Subset\Omega$. We first notice from \eqref{morrey} that for any $\tau\in(0,n)$, 
$$
\fint_{B_{2r}} |Du|\,dx \leq c_{\tau}r^{-\tau}\quad \text{for all\ }\ B_{2r}\subset \Omega'\ \ \text{with}\ \ 2r\in(0,r_0],
$$
where $r_0>0$ is from the proof of Theorem~\ref{mainthm1} and $c_\tau\geq 1$ depends on $n$, $p$, $q$, $L$, $r_0$ and $\tau$. Consider sufficiently small $2r<r_0$, which will be determined later. Let $v\in W^{1,\tphi}(B_{r})$ be the minimizer of 
\eqref{functionaltphiv} with $\tphi$ determined in \eqref{tphiphi}.
Then for $0<\rho<\frac r2$, applying \eqref{C1alpha0} with $B_{\rho}(x_0)=B_{r/2}$ and 
$\tau=\frac{2\rho}{r}$, Corollary~\ref{corcom} with 
$\omega(r)\lesssim r^{\beta}$ and \eqref{DvDu}, we have
\begin{align*}
\fint_{B_\rho}|Du-(Du)_{\rho}|\,dx 
&\leq 2\fint_{B_\rho}|Du-(Dv)_{\rho}|\,dx\\
& \leq 2\fint_{B_\rho}|Du-Dv|\,dx+2\fint_{B_\rho}|Dv-(Dv)_{\rho}|\,dx\\
& \lesssim \Big(\frac{r}{\rho}\Big)^n\fint_{B_{r}}|Du-Dv|\,dx
+\left(\frac{\rho}{r}\right)^{\alpha_0}\fint_{B_{r/2}}|Dv|\,dx\\
& \lesssim \left(r^{\gamma_0}\Big(\frac{r}{\rho}\Big)^n+
\left(\frac{\rho}{r}\right)^{\alpha_0}\right)\left(\fint_{B_{2r}}|Du|\,dx+1\right)\\
& \leq c_\tau r^{-\tau}\left(r^{\gamma_0}\Big(\frac{r}{\rho}\Big)^n+\left(\frac{\rho}{r}\right)^{\alpha_0}\right),
\end{align*}
where $\gamma_0:=\min\{\frac{p\beta}{2q^2},\gamma_1\}$. 
Finally, we choose
$\rho := r^{1+\gamma_0/(2n)}$ and 
$\tau:= \frac{\alpha_0\gamma_0}{4n}$. 
Suppose that $2r\leq \min\{r_0,4^{-2n/\gamma_0}\}$. 
Then $\rho<\frac{r}{2}$ and for the concentric balls $B_{\rho}\subset B_{2r}\subset \Omega'$ we have that 
\[
\fint_{B_\rho}|Du-(Du)_{\rho}|\,dx \lesssim  
c_{\tau} r^{\frac{\gamma_0}{2}-\tau}+r^{\tau}
\le 2c_{\tau}r^{\tau}=2c_{\tau}\rho^{\frac{\tau}{1+\gamma_0/(2n)}}.
\]
This yields that $Du\in C_{\loc}^{\alpha}(\Omega')$ with $\alpha=\frac{\alpha_0\gamma_0}{4n+2\gamma_0}$ by a Campanato type embedding, see for instance \cite[Chapter 3]{Gia1}. 
Since $\Omega'\Subset\Omega$ is arbitrary, we have the conclusion. 
The last inequality also yields an estimate for the semi-norm $[Du]_{C^\alpha(B_{r_0/2})}$, however, once 
we unravel the dependence from $c_\tau$, it is a somewhat complicated formula. 
\end{proof}

\begin{remark} By following the proofs, one can see that we use the condition \wMA{} only for fixed $\epsilon=\epsilon_0$ determined in \eqref{epsilon0}. Therefore, in Theorems~\ref{mainthm1} and \ref{mainthm2}, the condition \wMA{} can be replaced with the combination of \aone{} and \wMA{} with fixed $\epsilon>0$, where $\epsilon$ is sufficiently small and depends on $n$, $p$, $q$ and $L$. 
\end{remark}


\section{Examples of special structures}\label{sect:examples}

In this section, we show that our results include previous regularity results for special structures presented in the introduction. We provide details only for some of the 
cases, as the remaining ones can be handled by similar techniques. By $C^\omega$ we denote continuous functions with modulus of continuity $\omega$. 

\begin{corollary}[Perturbed autonomous case]\label{cor:PerturbedCase}
Let $a:\Omega\to [\nu,\Lambda]$ for some $0<\nu \leq\Lambda$, 
and let $\psi\in\Phic\cap C^1([0,\infty))$ with $\psi'$ satisfying 
\inc{p-1} and \adec{q-1} for some $1<p\leq q$. Define $\phi(x,t):=a(x)\psi(t)$. 
Then $\phi$ satisfies \MA{}, with $\omega(r)\approx \omega_a(2r)$, 
if and only if $a\in C^{\omega_a}$.
\end{corollary}
\begin{proof}
For any $B_r\subset\Omega$, 
\[
\begin{split}
\phi^{+}_{B_r}(t) 
= a^{+}_{B_r}\psi(t) 
= \left(1+\tfrac{a^+_{B_r}-a^-_{B_r}}{a^-_{B_r}}\right)\phi^-_{B_r}(t).
\end{split}
\]
Since $a^-_{B_r}\in [\nu,\Lambda]$, we obtain that 
\[
\frac{\phi^{+}_{B_r}(t) -\phi^{-}_{B_r}(t) }{\phi^{-}_{B_r}(t) }
\approx 
a^+_{B_r}-a^-_{B_r},
\]
and so the claim follows. 
\end{proof}

\begin{corollary}[Variable exponent case]
\label{cor:LpxCase}
Let $p:\Omega\to [p_1,p_2]$ for some $1<p_1\leq p_2$. Define $\phi(x,t):=t^{p(x)}$.
Then $\phi$ satisfies \MA{} if and only if there exists $\omega_p$ with 
\[
\lim_{r\to 0}\omega_p(r)\ln \tfrac1r=0
\quad\text{and}\quad p\in C^{\omega_p}.
\]
Moreover, $\phi$ satisfies \MA{} with $\omega(r)\lesssim r^\beta$ for some $\beta>0$ 
if and only if 
\[
\omega_p(r)\lesssim r^{\tilde\beta} 
\quad \text{for some }\ \tilde{\beta}>0.
\]
\end{corollary}
\begin{proof}
Fix $B_r\subset \Omega$ with $|B_r|\le 1$ and set $p^\pm=p^\pm_{B_r}$. 
Then we have $\phi^-_{B_r}(t)=t^{p^-}$ and $\phi^+_{B_r}(t) =t^{p^+}$ for $t\geq 1$ as well as  
$\phi^-_{B_r}(t)= t^{p^-}$ and $\phi^+_{B_r}(t)=t^{p^+}$ for $t<1$. 

Let us derive an equivalent form of the inequality in condition \MA{}. 
We may consider the range $[|B_r|, |B_r|^{-1}]$ in the condition, since 
it turns out that this choice of lower bound entails no additional 
restrictions in the variable exponent case. 
When $t\ge 1$, we have
\[
\phi^+_{B_r}(t) -\phi^-_{B_r}(t) = (t^{p^+-p^-}-1)t^{p^-} = (t^{p^+-p^-}-1)\phi^-_{B_r}(t).
\]
When $t\le 1$, the exponents $p^+$ and $p^-$ are interchanged. 
Since we consider the range 
\[
t\in \big[(\phi^-)^{-1}_{B_r}(|B_r|), (\phi^-)^{-1}_{B_r}(|B_r|^{-1})\big]
=
\big[|B_r|^{1/p^+}, |B_r|^{-1/p^-}\big],
\] 
we obtain that
\[
\sup_{t\in [|B_r|^{1/p^+}, |B_r|^{-1/p^-}]} \frac{\phi^+_{B_r}(t) -\phi^-_{B_r}(t)}{\phi^-_{B_r}(t)}
=  |B_r|^\frac{p^--p^+}{p^-}-1. 
\]
Suppose that $p\in C^{\omega_p}$. By the mean value theorem, $e^x -1\leq e^x x$. Thus 
\[
|B_r|^\frac{p^--p^+}{p^-}-1\le |B_r|^{-\omega_p(2r)}-1
\le  e^{n\omega_p(2r)\ln (1/r)}-1 
\le n e^{n\omega_p(2r)\ln (1/r)} \omega_p(2r)\ln \tfrac 1r=:\omega(r)
\]
and the inequality from \MA{} holds with this $\omega$. Moreover,
if $\lim_{r\to 0}\omega_p(r)\ln \tfrac1r=0$, then $\omega$ tends to zero, hence we obtain 
\MA{}. If $\omega_p(r)\lesssim t^\beta$, then $\omega$ is of order $\beta-\epsilon$ for 
any $\epsilon$.

Suppose next that $\phi$ satisfies \MA{} with function $\omega$. Then, for $r\le\frac12$, 
\[
|B_r|^\frac{p^--p^+}{p^-}-1\le \omega(r)
\]
hence
\[
p^+ - p^- 
\le 
-\frac{p^-\log(1+\omega(r))}{\log |B_r|} 
\lesssim 
\frac{p_2\log(1+\omega(r))}{n\log \frac1r}=:\omega_p(2r).
\]
Then  $\omega_p(2r)\log \frac1r \approx \log(1+\omega(r)) \to 0$. If $\omega(r)\le r^{\tilde \beta}$, then 
$\omega_p(2r) \approx \log(1+r^{\tilde \beta}) / \log \frac1r 
\lesssim r^{\beta}$.
\end{proof}

R\u{a}dulescu and colleagues \cite{CenRR18, RadRSZ_pp,ZhaR18} have considered 
a functional with model case $\phi(x,t) = t^{p(x)} + t^{q(x)}$, which 
they call ``double phase'' (it is different from the double phase functional
of Zhikov, considered below). 
To the best of our knowledge, this is the first regularity result this functional. 

\begin{corollary}[R\u{a}dulescu's double phase]
\label{cor:radulescu}
Let $p,q:\Omega\to [p_1,p_2]$ for some $1<p_1\leq p_2$ and 
$\phi(x,t) = t^{p(x)} + t^{q(x)}$.
Then $\phi$ satisfies \MA{} if there exist $\omega_m$ and $\omega_M$ 
with
\[
\lim_{r\to 0}\omega_m(r)=0, 
\quad
\min\{p,q\}\in C^{\omega_m},
\quad
\lim_{r\to 0}\omega_M(r)\ln \tfrac1r=0
\quad\text{and}\quad
\max\{p,q\}\in C^{\omega_M}.
\] 
In addition, $\phi$ satisfies \MA{} with $\omega(r)\le r^\beta$ for some $\beta>0$ 
if 
\[
\lim_{r\to 0}\omega_m(r)r^{-\tilde\beta}=0
\quad\text{and}\quad
\lim_{r\to 0}\omega_M(r)r^{-\tilde\beta}=0
\quad \text{for some }\tilde\beta>0.
\] 
\end{corollary}

This result can be proved with the same methods as Corollary~\ref{cor:LpxCase}; 
the details are left to the interested reader. 
Note that the regularity required of the minimum is lower 
than the regularity required of the maximum. This is due 
to the fact that we only require the inequality of \MA{} in the range 
$[\omega(r),1]$ where the minimum determines $\phi$, whereas the maximum is 
used in the range $[1,|B_r|]$. 

We now consider double phase problems in the sense of Zhikov and Mingione. 

\begin{corollary}[Double phase case]\label{corexdouble1}
Let $a \in C^{\omega_a}(\Omega)$ and $b \in C^{\omega_b}(\Omega)$ be non-negative with 
$0<\nu\leq a(\cdot)+b(\cdot)\leq \Lambda$ for some $0<\nu\leq \Lambda$, and 
$\psi,\xi\in\Phic\cap C^1([0,\infty))$ with $\psi', \xi'$ satisfying \azero{}, and \inc{p-1} and \dec{q-1} for some $1<p\leq q$. Suppose that $\frac \xi \psi$ is almost increasing.  
Define 
\[
\phi(x,t):=a(x)\psi(t)+b(x)\xi(t)
\] 
and, for $\epsilon\in [0,1)$, 
\[
\omega_\epsilon(r):= 
\omega_a(r)
+\omega_b(r)r^{n(1-\epsilon)} \xi\big(\psi^{-1}(r^{-n(1-\epsilon)})\big).
\]
If $\omega_\epsilon$ is bounded with $\lim_{r\to 0}\omega_\epsilon(r)=0$ 
when $\epsilon>0$, then 
$\phi$ satisfies \wMA{} with $\omega\approx \omega_\epsilon$.
\end{corollary}
\begin{proof}
Fix $B_r\subset\Omega$ so small that $\omega_a(2r),\omega_b(2r)\le \frac \nu 2$. 
Set $a^{\pm}:=a^{\pm}_{B_r}$, $b^{\pm}:=b^{\pm}_{B_r}$ and 
$\phi^{\pm}(t):=\phi^{\pm}_{B_r}(t).$ 
For $0\leq \epsilon<1$, suppose $t \in(0, t_2)$ with 
$t_2:= (\phi^-)^{-1}(|B_r|^{-1+\epsilon})$.  

We consider first $t>1$. 
Assume first that $b^-\ge \frac \nu4$. 
Then $\phi^-(t)\gtrsim \xi(t)\gtrsim \psi(t)$ and so 
\[
\phi^+(t)-\phi^-(t) \le (a^+-a^-)\psi(t) + (b^+-b^-) \xi(t) 
\lesssim (\omega_a(r)+\omega_b(r)) \phi^-(t). 
\]
We note that the case $b^-< \frac \nu4$ and $a^-< \frac \nu4$ cannot occur, 
since $a+b\ge \nu$ and $\omega_a(2r),\omega_b(2r)\le \frac \nu 2$.

Next, we consider $t>1$ and $a^-\ge \frac \nu4$. Then 
$\phi^-(t)\gtrsim \psi(t)$. Note that
$\phi(x,t)\approx \max\{a(x)\psi(t),b(x)\xi(t)\}$ and that by the continuity of the functions $a$ and $b$, there exists $x_t\in \overline{B_r}$ such that $\phi^-(t)=\phi(x_t,t)$.  Using these and that $\frac{\xi}{\psi}$ is almost increasing, we have 
\[
\frac{\xi(t)}{\phi^-(t)}
\approx
\min\Big\{ \frac{\xi(t)}{a(x_t)\psi(t)}, \frac{1}{ b(x_t)}\Big\}
\lesssim 
 \min\Big\{\frac{\xi(t_2)}{a(x_t)\psi(t_2) }, \frac{\xi(t_2)}{ b(x_t)\xi(t_2)}\Big\}
\lesssim 
\frac{\xi(t_2)}{\phi^-(t_2)}.
\]
Since $\frac{\psi(t)}{\phi^-(t)}\le \frac 4 \nu$ and 
$\psi(t_2)\lesssim a(x_{t_2})\psi(t_2) \lesssim \phi^{-}(t_2)
\approx r^{-n(1-\epsilon)}$, we conclude that 
\begin{align*}
\frac{\phi^+(t)-\phi^-(t)}{\phi^-(t)}
& \leq (a^+-a^-)\frac{\psi(t)}{\phi^-(t)} +(b^+-b^-)\frac{\xi(t)}{\phi^-(t)}\\
&\lesssim \omega_a(2r) +\omega_b(2r) \frac{\xi(t_2)}{\phi^-(t_2)} \\
&\lesssim \omega_a(r) +\omega_b(r) r^{n(1-\epsilon)} \xi(\psi^{-1}(r^{-n(1-\epsilon)})).
\end{align*}
We note that the factor multiplying $\omega_b(r)$ in the last expression is 
greater than $c>0$ depending on the parameters, so it can absorb the 
$+\omega_b(r)$ from the other cases to give $\omega_\epsilon$ in the statement of the result. 

The necessary inequality has been established for all cases when $t>1$. 
We next consider $t\le 1$. 
By \azero{} of $\psi$ and $\xi$, 
\[
\phi^+(t)-\phi^-(t) \le (a^+-a^-)\psi(t) + (b^+-b^-) \xi(t) \lesssim \omega_a(r)+\omega_b(r). 
\]
We use this as the additive term ``$+\omega(r)$'' in the definition of 
\wMA{} to cover small $t$. This concludes the proof of \wMA{}. 
\end{proof}

\begin{remark}
In the previous proof, we used the additive error ``$+\omega(r)$'' for \wMA{} 
to handle the case $t\le 1$. If $a\equiv 1$, then this is not needed, 
and we have also the following conclusion: 
if $\omega_0$ is bounded with $\lim_{r\to 0}\omega_0(r)=0$, then 
$\phi$ satisfies \MA{}.
\end{remark}

Suppose that $\xi(t)=\psi(t)\ln(e+t)$ and $a\equiv 1$ in 
Corollary~\ref{corexdouble1}. Then we have 
\[
\xi\big(\psi^{-1}(r^{-n(1-\epsilon)})\big)
=
r^{-n(1-\epsilon)} \ln\big(e+\psi^{-1}(r^{-n(1-\epsilon)})\big)
\approx
r^{-n(1-\epsilon)} \ln(e+\tfrac1r)
\]
since $\psi$ satisfies \inc{p} and \dec{q}. We see that the degenerate 
double phase functional satisfies \MA{} if $b$ is vanishing $\log$-H\"older continuous.

From Corollary~\ref{corexdouble1}, we obtain sharp regularity conditions 
for $\phi$ satisfying particular structures of double phase with power-functions.

\begin{corollary}\label{corexdouble2} 
Let $1<p\leq q$, $\beta\in(0,1]$, 
and $a\in C^{\omega_a}$ and $b\in C^{0,\beta}$ be non-negative. 
Define $\gamma_\epsilon:=\beta-\tfrac{n(q-p)(1-\epsilon)}{p}$, $\epsilon\geq 0$.
\begin{enumerate}
\item
Let $\phi(x,t)=t^p+b(x)t^q$.  \\
If $\frac qp< 1+\frac{\beta}{n}$, then $\phi$ satisfies \MA{} with 
$\omega(r)\approx r^{\gamma_0}$.\\
If  $\frac qp\le  1+\frac{\beta}{n}$, then 
$\phi$ satisfies 
\wMA{} with $\omega (r)\approx r^{\gamma_\epsilon}$.
\item
Let $\phi(x,t)=a(x)t^p+t^q$. 
Then $\phi$ satisfies \wMA{} with $\omega(r)\approx \omega_a(r)$.
\item
Let $a(x)t^p+b(x)t^q$ with $\nu\le a+b\le \Lambda$. 
If  $\frac qp\le 1+\frac{\beta}{n}$, then $\phi$ satisfies 
\wMA{} with $\omega(r)\approx r^{\gamma_\epsilon}+\omega_a(r)$.
\end{enumerate}
\end{corollary}


\appendix
\section{\texorpdfstring{$C^{1,\alpha}$}{C-1-alpha} Regularity for autonomous problems}
\label{app:a}

In this section, we prove Lemma \ref{lemHolder}. We follow the ideas in \cite{DieSV17,Le1}. 
In fact, it is almost enough to replace the map $t\mapsto t^p$ by the map 
$t\mapsto \phi(t)$ in the proof in \cite{Le1}. However, for completeness, we present the proof. Suppose $\phi\in C^1([0,\infty))\cap C^2((0,\infty))$ and $\phi'$ satisfies \inc{p-1} and \dec{q-1} for some $1<p\leq q$. We first consider the following non-degenerate problem 
for $\epsilon> 0$:
\begin{equation}\label{eqep}
\mathrm{div}\left(\frac{\phi_\epsilon(|Du_{\epsilon}|)}{|Du_\epsilon|}Du_{\epsilon}\right)=0\quad \text{in }\ \Omega, \quad \text{where }\ \phi_\epsilon (t):=\int_{0}^t \frac{\phi'(\epsilon+s)s}{\epsilon+s}\,ds,
\end{equation}
which is the Euler-Lagrange equation of the minimization problem
\begin{equation}\label{functionalepsilon}
\min_w \int_{\Omega}\phi_\epsilon (|Dw|)\,dx. 
\end{equation}
(In Lemma~\ref{lemHolder}, $\Omega=B_r$.) By the definition of $\phi_\epsilon$ we have 
$$
\frac{\phi_\epsilon'(t)}{t}=\frac{\phi'(\epsilon+t)}{\epsilon+t}, \quad \text{so that } \ \lim_{t\to 0^+}\frac{\phi_\epsilon'(t)}{t}=
\frac{\phi'(\epsilon)}{\epsilon}>0.
$$ 
Hence \eqref{eqep} is non-degenerate. We emphasize that all hidden constants in $\approx$ and $\lesssim$ in this appendix depend only on $n$, $p$ and $q$, but are independent of $\epsilon$. We observe by the first equality above and \inc{p-1} and \dec{q-1} of $\phi'$ that
\begin{equation}\label{Ap01}\begin{aligned}
\phi''_{\epsilon}(t)
&=\frac{\phi'(\epsilon+t)}{\epsilon+t}\left(1 +\left(\frac{\phi''(\epsilon+t)}{ (\epsilon+t)\phi'(\epsilon+t)}-1\right)\frac{t}{\epsilon+t}\right)\\
& \ge \frac{\phi'_\epsilon(t)}{t} \left(1+(p-2)\frac{t}{\epsilon+t}\right) \ge \min\{1,p-1\}\frac{\phi'_\epsilon(t)}{t} 
\end{aligned}\end{equation}
and
\begin{equation}\label{Ap02}
\phi''_{\epsilon}(t)\le \frac{\phi'_\epsilon(t)}{t} \left(1+(q-2)\frac{t}{\epsilon+t}\right)\le \max\{1,q-1\}\frac{\phi'_\epsilon(t)}{t}. 
\end{equation}
Therefore, $\phi_\epsilon'$ satisfies \inc{\min\{1,p-1\}} and \dec{\max\{1,q-1\}}, which implies that
\begin{equation}\label{Ap12}
t\phi_\epsilon''(t)\approx \phi_\epsilon'(t),\quad t\phi_\epsilon'(t)\approx \phi_\epsilon(t), \quad \frac{\phi_\epsilon (t)}{t^2}\approx\frac{\phi_\epsilon' (t)}{t} =\frac{\phi'(\epsilon+t)}{\epsilon+t}.
\end{equation}
In view of \cite{DieSV17}, in particular Lemmas~5.7 and 5.8, we have that $u_{\epsilon}\in W^{2,2}_{\loc}(\Omega)$ and $\phi_\epsilon(|Du_\epsilon|)\in W^{1,2}_{\loc}(\Omega)$ if $\epsilon>0$ and that for any $B_{2\rho}\Subset \Omega$ and $\epsilon\geq0$,
\begin{equation}\label{Appendix1}
\sup_{B_{\rho}} \phi_\epsilon (|Du_\epsilon|)
\lesssim
\fint_{B_{2\rho}} \phi_{\epsilon} (|Du_\epsilon|)\,dx.
\end{equation}
Here $u_0=u$ and $\phi_0=\phi$.

Fix $\epsilon>0$ and $B_{2\rho}\Subset \Omega$. From now on, for convenience, we shall simply write 
\begin{equation}\label{notationep}
u=u_\epsilon\quad\text{and}\quad v=\phi_\epsilon(|Du|)=\phi_\epsilon(|Du_\epsilon|).
\end{equation} 
We first notice from \eqref{eqep} and $u\in W^{2,2}_{\loc}(\Omega)$ that 
\[
\mathrm{div}\left(\frac{\phi_\epsilon'(|Du|)}{|Du|}Du\right) = \mathrm{div}\left(\frac{\phi'(\epsilon+|Du|)}{\epsilon+|Du|}Du\right)=\sum_{i,j=1}^n a_{ij} u_{x_ix_j}=0\qquad \text{a.e. in }\ \Omega,
\]
where $a_{ij}=a_{ij}(Du)$, $b_{ij}=b_{ij}(Du)$, 
$$
a_{ij}(z):=\frac{\phi'(\epsilon+|z|)}{\epsilon+|z|} b_{ij}(z)=\frac{\phi_\epsilon'(|z|)}{|z|} b_{ij}(z)
$$
and
$$
b_{ij}(z):=\left(\frac{\phi''(\epsilon+|z|)(\epsilon+|z|)}{\phi' (\epsilon+|z|)}-1\right)\frac{z_iz_j}{(\epsilon+|z|)|z|}+\delta_{ij}
$$
for $z\in\Rn$ ($\delta_{ij}$ is the kronecker delta, i.e.\ $\delta_{ij}=0$ if $i\neq j$ and $\delta_{ij}=1$ if $i=j$). As in \eqref{Ap01} and \eqref{Ap02} along with the fact that $\sum_{i,j}z_iz_j\eta_i\eta_j=(z\cdot \eta)^2$, we conclude that
\begin{equation}\label{ellipticityb}
\min\{1,p-1\} |\eta|^2\leq \sum_{i,j=1}^n b_{ij}(z)\eta_i\eta_j \leq \max\{1,q-1\} |\eta|^2\quad \text{for all }\ z,\eta\in\Rn.
\end{equation}

Consider the weak form of \eqref{eqep} and a unit vector $\nu\in S^{n-1}$. 
We see that 
\[
0=-\int_{\Omega}\frac{\phi_\epsilon'(|Du|)}{|Du|}Du\cdot D(
\zeta_\nu)\,dx
=\int_{\Omega}\sum_{i=1}^\infty \left(\frac{\phi_\epsilon'(|Du|)}{|Du|}u_{x_i}\right)_{\nu}
\zeta_{x_i}\,dx 
=\int_{\Omega}\sum_{i,j=1}^\infty a_{ij}u_{x_j\nu}\zeta_{x_i}\,dx
 \]
for any $\zeta\in C^\infty_0(\Omega)$, where the subscript $\nu$ indicates 
directional derivatives. Thus we have shown that 
\begin{equation}\label{etaeq}
\sum_{i,j=1}^n(a_{ij} u_{\nu x_j})_{x_i}=0
\end{equation}
in the weak sense. In addition, by the definition of $v$, cf.\ \eqref{notationep}, we have 
\[
v_{x_j}=\frac{\phi_\epsilon'(|Du|)}{|Du|}\sum_{k=1}^n u_{x_k}u_{x_kx_j},
\]
so that $b_{ij}v_{x_j}=\sum_{k=1}^na_{ij}u_{x_kx_j}u_{x_k}$. 
We conclude, with \eqref{etaeq} for the second equality, that
\begin{align*}
-\int_{B_\rho}\sum_{i,j=1}^nb_{ij}v_{x_j} \zeta_{x_i}\,dx 
&=-\sum_{k=1}^n\int_{B_\rho}\sum_{i,j=1}^na_{ij}u_{x_kx_j} (u_{x_k}\zeta)_{x_i} \,dx 
+ \int_{B_\rho}\sum_{i,j,k=1}^n a_{ij}u_{x_kx_j}u_{x_kx_i}\zeta \,dx\\
&=\int_{B_\rho}\sum_{i,j,k=1}^n a_{ij}u_{x_kx_j}u_{x_kx_i}\zeta \,dx
\end{align*}
for all $\zeta\in C^{\infty}_0(B_{\rho})$ with $B_\rho\Subset\Omega$. Therefore, we have 
\begin{equation}\label{operatorL}
\mathcal L v:= \sum_{i,j=1}^n(b_{ij}v_{x_i})_{x_j} = \sum_{i,j,k=1}^n a_{ij}u_{x_kx_j}u_{x_kx_i}=:g
\end{equation}
in the weak sense.
Moreover, by \eqref{ellipticityb} with $\eta=D(u_{x_k})$ and \eqref{Ap12}, we have 
\begin{equation}\label{gcondition}
 g\approx \frac{\phi_\epsilon(|Du|)}{|Du|^2} |D^2u|^2,
\end{equation}
where $|D^2u|^2 := \sum_{i,j}(u_{x_ix_j})^2$. 
In the same way as in \cite[Lemma 1]{Le1} with $v=\phi_\epsilon(|Du|)$, and $v^{-1/p}$ replaced by $[\phi_\epsilon^{-1}(v)]^{-1}=|Du|^{-1}$, we obtain 
for any $B_{4\rho}\Subset \Omega_r$, that
\begin{equation}\label{highg}
\fint_{B_{\rho}}g^{1+\delta}\, dx\lesssim\left(\fint_{B_{4\rho}}g\,dx\right)^{1+\delta}.
\end{equation}

Next, set 
$$
v_0:=v=\phi_\epsilon(|Du|)\quad \text{and}\quad v_k:=\frac{v}{\phi^{-1}_\epsilon(v)} u_{x_k}=\frac{\phi_\epsilon(|Du|)}{|Du|} u_{x_k},\ \ k=1,\dots,n.
$$
Then 
\begin{equation}\label{v0xj}
 v_{0,x_j} =v_{x_j}=\frac{\phi_\epsilon'(|Du|) Du_{x_j}\cdot Du}{|Du|},
\end{equation}
and, for $k=1,\dots n$, 
\begin{align*}
v_{k,x_j}
&=
\frac{v}{\phi_\epsilon^{-1}(v)}u_{x_kx_j}+ \left(1-\frac{v}{\phi_\epsilon^{-1}(v)\phi_\epsilon'(\phi_\epsilon^{-1}(v))}\right)\frac{v_{x_j}}{\phi_\epsilon^{-1}(v)}u_{x_k}\\
&=
\frac{\phi_\epsilon(|Du|)}{|Du|}u_{x_kx_j}+ \left(1-\frac{\phi_\epsilon(|Du|)}{|Du|\phi_\epsilon'(|Du|)}\right)\frac{\phi_\epsilon'(|Du|)Du_{x_j}\cdot Du}{|Du|^2}u_{x_k}.
\end{align*}
Here we note from \inc{\min\{2,p\}} and \dec{\max\{2,q\}} of $\phi_\epsilon$ that
$$
0<1-\frac{1}{\min\{2,p\}}\leq 
\underbrace{1-\frac{\phi_\epsilon(|Du|)}{|Du|\phi'_\epsilon(|Du|)}}_{=:A(|Du|)}
\leq 1-\frac{1}{\max\{2,q\}}<1.
$$
Then the previous expression for the partial derivatives implies that 
$v_k\in W^{1,2}_{\loc}(\Omega)$ since $u\in W^{2,2}_{\loc}(\Omega)$, $v=\phi_\epsilon(|Du|)\in W^{1,2}_{\loc}(\Omega)$ and $Du\in L^\infty_{\loc}(\Omega)$. Moreover, 
\begin{align*}
\sum_{k=1}^n |Dv_{k}|^2=\sum_{k,j=1}^n v_{k,x_j}^2 
&= \frac{\phi_\epsilon(|Du|)^2}{|Du|^2}|D^2u|^2
+A(|Du|)^2\frac{\phi_\epsilon'(|Du|)^2}{|Du|^4} \sum_{k,j=1}^n (Du_{x_j}\cdot Du)^2u_{x_k}^2\\
&\quad + 2 A(|Du|) \frac{\phi_\epsilon(|Du|)\phi_\epsilon'(|Du|)}{|Du|^3} 
\sum_{j=1}^n \left[Du_{x_j}\cdot Du \sum_{k=1}^n u_{x_kx_j} u_{x_k}\right].
\end{align*}
Since
$A(|Du|)\approx 1$, $\sum_{k,j}(Du_{x_j}\cdot Du)^2u_{x_k}^2\le |Du|^4|D^2u|^2$
and 
$Du_{x_j}\cdot Du \sum_{k} u_{x_kx_j} u_{x_k}=(Du_{x_j}\cdot Du)^2\geq0$,
we obtain
\begin{equation}\label{gv}
gv \approx\frac{\phi_\epsilon(|Du|)^2}{|Du|^2}|D^2u|^2 \leq \sum_{k=1}^n |Dv_{k}|^2\lesssim \frac{\phi_\epsilon(|Du|)^2}{|Du|^2}|D^2u|^2\approx gv.
\end{equation}
Using the expression for the partial derivative $v_{k,x_j}$ and \eqref{etaeq}, we see that for $k=1,2,\dots,n$,
\begin{equation}\label{Ap13}\begin{aligned}
\mathcal L v_k
&= 
\sum_{i,j=1}^n (b_{ij}v_{k,x_j})_{x_i}\\
&= 
\sum_{i,j,l=1}^n \left(a_{ij} u_{x_lx_j} \left[\frac{|Du|v\delta_{lk}}{\phi_\epsilon'(|Du|) |Du|} + \frac{ u_{x_k}u_{x_l}}{|Du|}-\frac{|Du|vu_{x_k}u_{x_l}}{\phi_\epsilon'(|Du|) |Du|^3}\right]\right)_{x_i}\\
&= 
\frac{\phi_\epsilon'(|Du|)}{|Du|} \sum_{i,j,l=1}^n b_{ij} u_{x_lx_j} \left( \frac{ u_{x_k}u_{x_l}}{|Du|}+ \frac{v\delta_{lk}}{\phi_\epsilon'(|Du|)} -\frac{vu_{x_k}u_{x_l}}{\phi_\epsilon'(|Du|) |Du|^2}\right)_{x_i} \\
&=: g_k
\end{aligned}\end{equation}
in the weak sense for test functions in $C^\infty_0(B_{\rho})$. 
Note that $g_k$ is formulated in terms of first and second partial derivatives of $u$.
We use the estimates $|u_{x_j}|\le |Du|$ and $|u_{x_i x_j}|\le |D^2 u|$ to 
conclude that 
\[
\left| \left(\frac{ u_{x_k}u_{x_l}}{|Du|}\right)_{x_i}\right|= \left|\frac{ u_{x_kx_i}u_{x_l}+u_{x_k}u_{x_lx_i}}{|Du|}-\frac{ u_{x_k}u_{x_l}\sum_{m=1}^nu_{x_mx_i}}{|Du|^2}\right|
\lesssim |D^2u|. 
\]
Similarly, using also $t \phi_\epsilon'(t) \approx \phi_\epsilon(t)$ from \eqref{Ap12}, 
we estimate the other multipliers of $b_{ij} u_{x_lx_j}$ by $|D^2u|$, as well. 
Since $b_{ij}\approx 1$ by \eqref{ellipticityb}, we conclude by \eqref{gcondition} 
that 
\begin{equation}\label{Ap14}
|g_k| 
\lesssim
\left| \frac{\phi_\epsilon'(|Du|)}{|Du|} \sum_{i,j,l=1}^n b_{ij}u_{x_lx_j} 
|D^2u| \right|\lesssim \frac{\phi_\epsilon'(|Du|)}{|Du|} |D^2u|^2 \approx g. 
\end{equation}

From now on, fix $B_{32s}\Subset \Omega$. For $k=0,1,\dots,n$, let $h_k\in W^{1,2}(B_{s})$ be a weak solution to
$$
\mathcal L h_k=(b_{ij} h_{k,x_j})_{x_i} =0\quad \text{in }\ B_s,\quad\quad h_k=v_k\quad \text{on }\ \partial B_{s},
$$
and let
$$
w_k=h_k-v_k\in W^{1,2}_0(B_{s}).
$$
Then, by De Giorgi's theory for linear equation, see for instance \cite[Theorem 7.7]{Gi1}, we have that, for any concentric balls $B_\rho\subset B_s$ with $0<\rho\leq s$,
\begin{equation}\label{Ap15}
\fint_{B_\rho}|Dh_k|^2\,dx \lesssim \left(\frac{\rho}{s}\right)^{\beta-2}\fint_{B_{s}} |Dh_k|^2\,dx \lesssim \left(\frac{\rho}{s}\right)^{\beta-2}\fint_{B_{s}} |Dv_k|^2\,dx
\end{equation}
for some $\beta=\beta(n,p,q)\in(0,2)$. In addition, by \eqref{ellipticityb}, 
$\mathcal L h_k=0$, \eqref{operatorL}, \eqref{Ap13} and \eqref{Ap14},
\begin{align*}
\int_{B_{s}}|Dw_k|^2\,dx & \lesssim \int_{B_{s}}\sum_{i,j=1}^n b_{ij}(h_k-v_k)_{x_j} w_{k,x_i}\,dx\\
 &= -\int_{B_{s}}\sum_{i,j=1}^n b_{ij}v_{k,x_j} w_{k,x_i}\,dx= \int_{B_{s}}g_k w_k\,dx
\lesssim \int_{B_{s}}g |w_k|\,dx.
\end{align*}
Here we interpret $g_0:=g$. (Note that $w_k\not\in C^\infty_0(B_\rho)$, but we can use $w_k$ as a test function by an approximation argument.)

Hence applying H\"older's inequality and \eqref{highg} we have that 
$$
\fint_{B_{s}}|Dw_k|^2\,dx\lesssim \left(\fint_{B_{4s}}g\,dx \right) \left(\fint_{B_{s}}|w_k|^{1+1/\delta}\,dx \right)^{\frac{1}{1+1/\delta}}.
$$
Furthermore, the same arguments used to prove \cite[(3.8) and (3.13)]{Le1} (here we need 
\eqref{v0xj} and \eqref{gv}) yield that 
$$
\fint_{B_{4s}}g\,dx \lesssim \frac{1}{M(4s)}\sum_{k=0}^n\fint_{B_{8s}}|Dv_k|^2\,dx,\quad \text{where }\ M(\rho) :=\sup_{B_{\rho}} v,
$$
and 
$$
\frac{1}{M(4s)}\left(\fint_{B_s}|w_k|^{1+1/\delta}\right)^{\frac{1}{1+\delta}}\leq c\left(1-\frac{M(\rho)}{M(8s)}\right)^{\frac{1}{1+\delta}}\quad \text{for all }\ \rho\in(0, s].
$$
Therefore, combining the last three estimates and \eqref{Ap15} we have, for $\rho\in(0,s]$,
that 
$$
\sum_{k=0}^n\int_{B_\rho} |Dv_k|^2\, dx 
\lesssim\left(\left(\frac{\rho}{s}\right)^{n-2+\beta}+\left(1-\frac{M(\rho)}{M(8s)}\right)^{\frac{1}{1+\delta}} \right)\sum_{k=0}^n\int_{B_{8s}} |Dv_k|^2\, dx.
$$
Finally by a standard iteration argument as in \cite[p. 857]{Le1} and Poincar\'e's inequality, we can find $\beta_1\in(0,\beta)$ such that 
$$
\sum_{k=0}^n\fint_{B_\rho}|v_k-(v_k)_{B_\rho}|^2\,dx \lesssim \rho^2\sum_{k=0}^n\fint_{B_\rho} |Dv_k|^2\, dx
\lesssim \left(\frac{\rho}{s}\right)^{\beta_1}M(4s)^2
$$
for any $\rho\in(0,s]$. With the definition of $v_k$, this implies that, for any $x,y\in B_s$,
$$
\Big|\frac{\phi_{\epsilon}(|Du(x)|)}{|Du(x)|}u_{x_k}(x)-\frac{\phi_{\epsilon}(|Du(y)|)}{|Du(y)|}u_{x_k}(y)\Big| 
= |v_k(x)-v_k(y)| 
\lesssim \left(\frac{|x-y|}{s}\right)^{\frac{\beta_1}{2}}M(8s).
$$
We use Proposition~\ref{prop000}(1) with $\phi_{\epsilon}$ in place of $\phi'$
to conclude that 
\[
\Big|\frac{\phi_{\epsilon}(|z_1|)}{|z_1|}z_1-\frac{\phi_{\epsilon}(|z_2|)}{|z_2|}z_2\Big| 
\gtrsim 
\frac{\phi_{\epsilon}(|z_1|+|z_2|)}{|z_1|+|z_2|} |z_1-z_2|
\ge 
\phi_{\epsilon}(|z_1-z_2|)
\]
where we used $|z_1|+|z_2|\ge |z_1-z_2|$ and \inc{1} of $\phi_\epsilon$ in the last step. 
Applying this in the previous estimate with $z_1=Du(x)$ and $z_2=Du(y)$, we find that 
\[
\phi_{\epsilon}(|Du(x)-Du(y)|) 
\lesssim \left(\frac{|x-y|}{s}\right)^{\frac{\beta_1}{2}}M(8s).
\]

We now undo the convention of omitting $\epsilon$ from \eqref{notationep} for the final part. 
Inserting \eqref{Appendix1} with the definition of $M(\rho)$ into the above estimate, 
we have that 
$$
\phi_\epsilon\big(|Du_\epsilon(x)-Du_\epsilon(y)|\big)
\lesssim 
\left(\frac{|x-y|}{s}\right)^{\frac{\beta_1}{2}}\fint_{B_{16s}}\phi_\epsilon(|Du_\epsilon|)\,dx.
$$
At this point, we restrict our attention to the case $\Omega=B_{32s}$ and 
consider minimizers $u_\epsilon$ of \eqref{functionalepsilon} with the boundary 
value restriction $w\in u+W^{1,\phi_\epsilon}_0(B_{32s})$. We apply $\phi_\epsilon^{-1}$ to both sides and use \dec{\max\{2,q\}} of $\phi_\epsilon$, to get that
$$
|Du_\epsilon(x)-Du_\epsilon(y)|\lesssim \left(\frac{|x-y|}{s}\right)^{\alpha_0}\phi_\epsilon^{-1}\left(\fint_{B_{16s}}\phi_\epsilon(|Du|)\,dx \right)
$$
for some $\alpha_0\in(0,1)$. Letting $\epsilon\to 0$, we can remove $\epsilon$ in the above estimate as in the proof of \cite[Lemma 4.9]{DieSV17}. 
Finally, by the same argument as in the proof of Lemma~\ref{reverse} with \eqref{Appendix1} and Jensen's inequality for the concave function equivalent to $\phi^{1/q}$ (see Proposition~\ref{prop00}(2)) we also see that
\begin{equation}\label{AApf1}
\fint_{B_{16s}}\phi(|Du|)\,dx \lesssim \left(\fint_{B_{32s}}\phi(|Du|)^{\frac{1}{q}}\,dx\right)^q\lesssim \phi\left(\fint_{B_{32s}} |Du|\,dx\right).
\end{equation}
These imply, for any $x,y\in B_s$ and $B_{32s}\Subset \Omega$, that
$$
|Du(x)-Du(y)|\leq c \left(\frac{|x-y|}{s}\right)^{\alpha_0}\fint_{B_{32s}}|Du|\,dx,
$$ 
which shows \eqref{C1alpha0}. In addition, from \eqref{Appendix1} and \eqref{AApf1}, we also have \eqref{Lip}. 


\section{Weighted estimate for autonomous problems}\label{app:b}

In this appendix, we discuss the global weighted estimate \eqref{weightedestimate}. 
For global regularity estimates, the regularity of the boundary of the domain is a 
delicate issue. In particular, the Reifenberg flat condition is considered sharp 
for Calder\'on--Zygmund type estimates for problems in divergence form. 
Hence we shall give a result for domains satisfying the this condition. 
We say that a bounded domain $\Omega$ is \textit{$(\delta,R)$-Reifenberg flat} for 
some small $\delta\in (0,1)$ and $R>0$ if for any $y\in\partial\Omega$ and $r\in(0,R]$ 
there exists an isometric coordinate system with the origin at $y$, say $(x_1,\dots,x_n)$, such 
that in this coordinate system, 
$$
B_r(0)\cap \{x_n>\delta r\} \ \subset\ \Omega\cap B_r(0)\ \subset\ B_r(0)\cap \{x_n>-\delta r\}. 
$$
Note that a domain with Lipschitz boundary with Lipschitz semi-norm $\delta\in(0,1)$ 
is $(\delta,R)$-Reifenberg flat for some $R>0$ and that the ball $B_r$ is $(\delta,2\delta r)$-Reifenberg flat for any $\delta\in (0,\frac{1}{2})$.

For $1\leq s\leq \infty$, let $A_s$ be the Muckenhoupt class. 
In particular, for $1< s< \infty$, a weight $w$ (i.e., $w\in L^1_{\loc}(\Rn)$ and $w\geq 0$) is an $A_s$-weight, $w\in A_s$, if 
$$
[w]_{A_s}:=\sup_{B\subset \Rn} \left(\fint_B w\,dx\right)\left(\fint_{B}w^{-\frac{1}{s-1}}\,dx\right)^{s-1} <\infty.
$$ 
For the properties of the $A_s$ class, we refer to \cite{Gra1}. 

\begin{theorem}\label{thmAB}
Let $\phi\in \Phic\cap C^1([0,\infty))\cap C^2((0,\infty))$ with $\phi'$ satisfying \inc{p-1} and \dec{q-1} for some $1<p\leq q<\infty$, and let $w\in A_{s}$ for some $s\in (1,\infty)$. There exists a small $\delta=\delta(n,p,q,s,[w]_{A_s})\in(0,1)$ such that if $\Omega$ is $(\delta,R)$-Reifenberg flat for some $R>0$, $v_0\in W^{1,\phi}(\Omega)$ satisfies $\phi(|Dv_0|)\in L^s_w(\Omega)$ and $v\in W^{1,\phi}(\Omega)$ is the weak solution to 
\begin{equation}\label{weqvv0}
\div \left(
\frac{\phi'(|Dv|)}{|Dv|}Dv\right)=0\quad\text{in}\ \ \Omega\quad\text{with}\quad v=v_0 \ \ \text{on}\ \partial \Omega,
\end{equation}
then 
$$
\int_{\Omega}\phi(|Dv|)^s\, w\, dx \leq c \int_{\Omega}\phi(|Dv_0|)^s\,w\, dx
$$
for some $c=c(n,p,q,s,[w]_{A_s},\frac{\diam(\Omega)}R)>0$. In particular, 
letting $\Omega=B_r$ we have \eqref{weightedestimate}, since 
$B_r$ is $(\delta,2\delta r)$-Reifenberg flat.
\end{theorem}

\begin{remark}
In Theorem \ref{thmAB}, $\delta$ is decreasing as a function of $[w]_{A_s}$, 
see \cite[Remark~2.2]{MP1}. Moreover, we can also see by analyzing the proof that the 
constant $c$ is increasing in $[w]_{A_s}$ and $\frac{\diam(\Omega)}R$ when the other 
is constant. Therefore, when $\Omega=B_r$, the constant $c$ is increasing in $[w]_{A_s}$, 
since $\frac{\diam(\Omega)}{R}=\frac{1}{2\delta}$.
\end{remark}

\begin{proof}[Sketch of the proof of Theorem~\ref{thmAB}]
For the $p$-Laplacian case, that is, $\phi(t)=t^p$, the weighted estimate has been proved in \cite{MP1}, see also \cite{BR1}, for the following equation:
$$
\div \left(|Dv|^{p-2}Dv\right)=\div \left(|F|^{p-2}F\right)\quad\text{in}\ \ \Omega\quad\text{with}\quad v=0 \ \ \text{on}\ \partial \Omega.
$$
Specifically, in \cite{MP1}, it has been shown that for the above equation, 
$$
\int_{\Omega}|Dv|^{ps}w\, dx \leq c \int_{\Omega}|F|^{ps}w\, dx
$$
for any $w\in A_s$ and any $F\in L^{ps}_w(\Omega,\Rn)$.
Moreover, it turns out that this result without a weight (i.e., $w\equiv 1$) is naturally extended \cite{BC1} to the equation involving a general function $\phi$ 
\begin{equation}\label{PhiLaplaceF}
\div \left(
\frac{\phi'(|Dv|)}{|Dv|}Dv\right)=\div \left(
\frac{\phi'(|F|)}{|F|}F\right)\quad\text{in}\ \ \Omega\quad\text{with}\quad v=0 \ \ \text{on}\ \partial \Omega.
\end{equation}
Therefore, proceeding as in \cite{MP1} with minor modification, one can prove that for the equation \eqref{PhiLaplaceF},
$$
\int_{\Omega}\phi(|Dv|)^{s}w\, dx \leq c \int_{\Omega}\phi(|F|)^{s}w\, dx
$$
for any $w\in A_s$ and any $F\in L^{\phi}(\Omega,\Rn)$ satisfying $\phi(|F|)\in L^s_w(\Omega)$.

In this theorem, we consider non-zero boundary data $v_0$. However, the gradient of $v_0$ can be handled in a similar way as for $F$ in the results mentioned above. Hence, by the same argument as in \cite{MP1}, replacing $t^p$ by $\phi(t)$ and changing boundary comparison estimates from \cite[Lemma 4.6]{MP1} to Lemmas~\ref{wcomlem1} and \ref{wcomlem2} below, we have the desired estimate. 
\end{proof}



For the rest of the paper, we suppose the assumptions of Theorem~ \ref{thmAB}. We consider 
our problem \eqref{weqvv0} on a local region near the boundary of $\Omega$. 
Define $\Omega_r(x):=\Omega\cap B_r(x)$, $B_r:=B_r(0)$, $\Omega_r=\Omega_r(0)$, 
$B^+_r:=B_r\cap\{x_n>0\}$ and $B^-_r:=B_r\cap\{x_n<0\}$. Then we consider 
our equation in the region $\Omega_{5r}$ satisfying that $5r<R$ and 
$$
B_{5r}^{+}\ \subset\ \Omega_{5r} \ \subset\ 
B_{5r}\cap \{x_n>-10\delta r\},
$$
Here, $\delta\in(0,1)$ and $R>0$ come from the $(\delta,R)$-Reifenberg flat condition of $\Omega$ and so $\delta$ has to be determined later and $R$ is given. 
Note that in view of the scaling invariance property of \eqref{weqvv0}, see for instance the proof of Lemma~\ref{lemCZ}, we may let $r=1$ and consider assumption \eqref{wcomlem1ass} below.

We first compare our equation \eqref{weqvv0} with an equation having zero boundary 
values on $\partial\Omega$ in a local region near the boundary. 

\begin{lemma}\label{wcomlem1}
For $v_0\in W^{1,\phi}(\Omega)$ let $v\in W_0^{1,\phi}(\Omega)$ be a weak solution to 
\eqref{weqvv0}. For any $\epsilon>0$ there exists small $\delta\in(0,1)$ depending on $n$, $p$, $q$ and $\epsilon$ such that if $\Omega$ is $(\delta,5)$-Reifenberg flat and 
\begin{equation}\label{measuredensity}
B_5^{+}\ \subset\ \Omega_5 \ \subset\ 
B_5\cap \{x_n>-10\delta\},
\end{equation}
\begin{equation}\label{wcomlem1ass}
\fint_{\Omega_5} \phi(|Dv|)\,dx \leq 1
\quad\text{and}\quad
\fint_{\Omega_5} \phi(|Dv_0|)\,dx\leq \delta,
\end{equation}
then for the weak solution $w\in W^{1,\phi}(\Omega_5)$ to
\begin{equation}\label{wweq}
\div \left(
\frac{\phi'(|Dw|)}{|Dw|}Dw\right)=0\quad\text{in}\ \ \Omega_{5}\quad\text{and}\quad w=v-v_0 \ \ \text{on}\ \partial\Omega_{5},
\end{equation}
we have 
\begin{equation}\label{westimate1}
\fint_{\Omega_5}\phi(|Dw|)\,dx\leq c\quad\text{and}\quad
\fint_{\Omega_5}\phi(|Dv-Dw|)\,dx\leq \epsilon.
\end{equation} 
Here, $c>0$ depends on $p$ and $q$, but is independent of $\epsilon$.
\end{lemma}

\begin{proof}
Since $w-v+v_0\in W^{1,\phi}_0(\Omega_5)$, we have by \eqref{wweq} and \eqref{weqvv0} that
\begin{equation}\label{wlemcom1pf1}
\int_{\Omega_5} \frac{\phi'(|Dw|)}{|Dw|}Dw \cdot D(w-v+v_0)\,dx= 0 =
\int_{\Omega_5} \frac{\phi'(|Dv|)}{|Dv|}Dv \cdot D(w-v+v_0)\,dx.
\end{equation}
In view of $\phi(t)\le t\phi'(t)$ and Propositions~\ref{prop00}(5) and \ref{prop0}(4), 
the first equality above implies that
\begin{align*}
\fint_{\Omega_5}\phi(|Dw|)\,dx
\leq 
\fint_{\Omega_5} \phi'(|Dw|)|Dv-Dv_0|\,dx
\leq 
\fint_{\Omega_5} [\tfrac{1}{2}\phi(|Dw|)\, dx+c\phi(|Dv-Dv_0|) ]\,dx
\end{align*}
for some $c=c(p,q)>0$. By \eqref{wcomlem1ass}, we have the first estimate in \eqref{westimate1}.

We next prove the second estimate in \eqref{westimate1}. 
By Proposition~\ref{prop000}(1), \eqref{wlemcom1pf1} and Propositions~\ref{prop00}(5)
and \ref{prop0}(4) we have that for $\kappa_1\in(0,1)$,
\begin{align*}
\fint_{\Omega_5} \frac{\phi'(|Dw|+|Dv|)}{|Dw|+|Dv|}|Dw-Dv|^2\,dx&\lesssim
\fint_{\Omega_5} \left(\frac{\phi'(|Dw|)}{|Dw|}Dw-\frac{\phi'(|Dv|)}{|Dv|}Dv\right) \cdot D(w-v)\,dx\\
&\hspace{-1cm}\lesssim\fint_{\Omega_5} (\phi'(|Dw|)+\phi'(|Dv|)) | Dv_0|\,dx\\
&\hspace{-1cm}\lesssim \kappa_1\fint_{\Omega_5} [\phi(|Dw|)+\phi(|Dv|)] \,dx+ \frac{1}{\kappa_1^{q-1}}\fint_{\Omega_5} \phi(|Dv_0|)\,dx.
\end{align*}
Moreover, by Proposition \ref{prop000}(3), for any $\kappa_2\in(0,1)$,
\[
 \phi(|Dw-Dv|)\leq \kappa_2(\phi(Dw)+\phi(Dv))+\kappa_2^{-1} \frac{\phi'(|Dw|+|Dv|)}{|Dw|+|Dv|}|Dw-Dv|^2.
\]
Combining the above two estimates we have
\begin{align*}
\fint_{\Omega_5} \phi(|Dw-Dv|)\,dx 
&\lesssim \kappa_2\fint_{\Omega_5}[\phi(Dw)+\phi(Dv)]\,dx+ \frac{\kappa_1}{\kappa_2}\fint_{\Omega_5} \phi(|Dw|)+\phi(|Dv|) \,dx\\
&\quad + \frac{1}{\kappa_2\kappa_1^{q-1}}\fint_{\Omega_5} \phi(|Dv_0|)\,dx.
\end{align*}
Finally applying \eqref{wcomlem1ass} and the first estimate in \eqref{westimate1} and choosing sufficiently small numbers $\kappa_1$, $\kappa_2$ and $\delta$ depending $n$, $p$, $q$ on $\epsilon$, we have the second estimate in \eqref{westimate1}.
\end{proof}

We also notice that the weak solution $w$ to \eqref{wweq} has value zero on 
$\partial\Omega_5\cap B_5$. 
We next compare \eqref{wweq}, which assumes zero boundary values on $\partial\Omega_5\cap B_5$, with an equation defined in $B_2^+$ with zero boundary 
values on $B_{2}\cap\{x_n=0\}.$ A similar result can be found in \cite[Lemma 3.6]{BC1}. The proof of that lemma employs a compactness argument. Here we give a more direct approach which clearly shows the dependence on $\delta$. 
 
\begin{lemma}\label{wcomlem2}
Let $\eta=\eta(x_n) \in C^\infty(\R)$ with $\eta=0$ if $x_n\leq 0$, $\eta=1$ 
if $x_n\geq \delta $ and $|\eta'|\leq \frac 2\delta$. For any $\epsilon>0$ there exists a 
small $\delta>0$ depending on $n$, $p$, $q$ and $\epsilon$, such that, under the assumptions of the above lemma, if $w_0$ is the weak solution to
\[
\div \left(
\frac{\phi'(|Dw_0|)}{|Dw_0|}Dw_0\right)=0\quad\text{in}\ \ B^+_{2}\quad\text{and}\quad w_0=\eta w\ \ \text{on}\ \partial B_{2}^+,
\]
then 
\begin{equation}\label{westimate2}
\fint_{B_2^+}\phi(|Dw_0|)\,dx\leq c\quad\text{and}\quad
\fint_{B_2^+}\phi(|Dw-Dw_0|)\,dx\leq \epsilon.
\end{equation} 
Moreover, 
\begin{equation}\label{bdLip}
\|\phi(|Dw_0|)\|_{L^\infty(\Omega_1)}=\|\phi(|Dw_0|)\|_{L^\infty(B_1^+)}\leq c \fint_{B^+_2}\phi(|Dw_0|)\, dx \leq c,
\end{equation}
where we extend $w_0$ by zero to $B^-_2$. Here constants $c$ depend on $n$, $p$ and $q$, but are independent of $\epsilon$.
\end{lemma}
\begin{proof}
We follow the technique in \cite[Lemma 2.5]{KR1}, see also \cite[Lemma 2.5]{BOY1}. Clearly, 
$(\delta,R)$-Reifenberg flat domains with $\delta\in(0,\frac12)$ satisfy the measure density condition $|B_r|4^{-n}\leq |\Omega_r(x)|\leq |B_r|$ and $4^{-n}|B_r|\leq |B_r(x)\setminus\Omega_r(x)|$ for all $x\in\partial \Omega$ and $r\in(0,R]$. 
One can show as in \cite[Theorem 3.9]{Ok0} that, for equation \eqref{wweq}, 
there exists $\sigma=\sigma(n,p,q)\in(0,1)$ such that $\phi(|Dw|)\in L^{1+\sigma}_{loc}(B_5)$ (we extend $w$ by $0$ in $B_5\setminus\Omega_5$) and 
$$
\left(\fint_{\Omega_3}\phi(|Dw|)^{1+\sigma}\,dx\right)^{\frac{1}{1+\sigma}}
\lesssim \fint_{\Omega_4}\phi(|Dw|)\,dx.
$$
Then by H\"older's inequality with \eqref{measuredensity}, we observe that
\begin{equation}\label{wlemcom2pf1}\int_{\Omega_3\cap\{x_n\leq \delta\}}\phi(|Dw|)\, dx \lesssim \delta^{\frac{\sigma}{1+\sigma}} \left(\int_{\Omega_3}\phi(|Dw|)^{1+\sigma}\, dx\right)^{\frac{1}{1+\sigma}}\lesssim \delta^{\frac{\sigma}{1+\sigma}}\int_{\Omega_4}\phi(|Dw|)\, dx.
\end{equation}
In addition, using the fact that $w\equiv 0$ in $B_4\setminus \Omega_4$ and 
$w$ is absolutely continuous on almost all lines parallel to the co-ordinate axes, 
as well as Jensen's inequality, we find that  
\begin{equation}\label{wlemcom2pf2}
\begin{aligned}
\int_{\Omega_2\cap\{x_n\leq \delta\}}\phi(|D\eta||w|)\, dx 
&\lesssim \int_{\Omega_2\cap\{x_n\leq \delta\}}\phi\left(\frac{1}{\delta}\left|\int^{x_n}_{-8\delta}D_nw(x',y)\,dy\right|\right)\, dx\\
&\lesssim \int_{\Omega_2\cap\{x_n\leq \delta\}}\phi\left(\fint^{\delta}_{-8\delta}|D_nw(x',y)|\,dy\right)\, dx\\
&\lesssim \int_{\{|x'|\leq 2\}\times\{-10\delta<x_n\leq \delta\}}\fint^{\delta}_{-8\delta} \phi\left(|D_nw(x',y)|\right)\, dy\, dx'dx_n\\
&\lesssim \int_{\Omega_3\cap\{(x',y):y\leq \delta\}} \phi\left(|Dw(x',y)|\right)\, dx'dy.
\end{aligned}\end{equation}
In the last inequality above, we used the facts that $\{|x'|\leq 2\}\cap \{|x_n|\leq 2\}\subset B_3$ and $Dw\equiv0$ in $B_3\setminus \Omega_3$.

Since $w_0-\eta w\in W^{1,\phi}_0(B_2^+)$ and $w_0$ is a minimizer, 
using also \eqref{wlemcom2pf2}, we have that
\begin{equation}\label{wlemcom2pf3}\begin{aligned}
\int_{B_2^+}\phi(|Dw_0|)\, dx &\le \int_{B_2^+} \phi(|D(\eta w)|)\,dx \\
&\lesssim \int_{B_2^+} \phi(|Dw|)\,dx+ \int_{B_2^+\cap\{x_n\leq \delta\}} \phi(|D\eta| |w|)\,dx \\
&\lesssim \int_{\Omega_3} \phi(|Dw|)\,dx,
\end{aligned}\end{equation}
which together with \eqref{measuredensity} and \eqref{westimate1} yields the first estimate in \eqref{westimate2}.

We next derive the second estimate in \eqref{westimate2}. Since $w_0-\eta w\in W^{1,\phi}_0(B_2^+)\cap W^{1,\phi}_0(\Omega_2)$ we have 
$$
\int_{B^+_2} \left(\frac{\phi'(|Dw_0|)}{|Dw_0|}Dw_0-
\frac{\phi'(|Dw|)}{|Dw|}Dw\right) \cdot D(w_0-\eta w)\, dx =0
$$
which together with Propositions~\ref{prop00}(5) and \ref{prop0}(4) 
implies that for any $\kappa_1\in(0,1)$
\begin{align*}
&\int_{B^+_2} \left(\frac{\phi'(|Dw_0|)}{|Dw_0|}Dw_0-
\frac{\phi'(|Dw|)}{|Dw|}Dw\right) \cdot D(w_0-w)\, dx\\
& = \int_{B^+_2\cap\{x_n\leq \delta \}} \left(\frac{\phi'(|Dw_0|)}{|Dw_0|}Dw_0- \frac{\phi'(|Dw|)}{|Dw|}Dw\right) \cdot D(\eta w- w)\, dx\\
& \leq \int_{B^+_2\cap\{x_n\leq \delta \}} \left(\phi'(|Dw_0|)+ \phi'(|Dw|)\right) \left(|D\eta| |w|+ |Dw|\right)\, dx\\
& \lesssim \kappa_1 \int_{B^+_2\cap\{x_n\leq \delta \}} \phi(|Dw_0|)+ \phi(|Dw|) \, dx+ 
\frac{1}{\kappa_1^{q-1}}\int_{B^+_2\cap\{x_n\leq \delta \}} \phi(|D\eta| |w|)+ \phi(|Dw|)\, dx.
\end{align*}
Applying Proposition~\ref{prop000}(3) and \eqref{wlemcom2pf1}--\eqref{wlemcom2pf3}, we see that for any $\kappa_2\in(0,1)$,
$$
\int_{B^+_2}\phi(|Dw_0-Dw|)\, dx \lesssim \left(\kappa_2+\frac{\kappa_1}{\kappa_2}\right) \int_{\Omega_3}\phi(|Dw|)\, dx+ \frac{\delta^{\frac{\sigma}{1+\sigma}} }{\kappa_2\kappa_1^{q-1}} \int_{\Omega_4}\phi(|Dw|)\, dx.
$$
Therefore, using the first estimate in \eqref{westimate1} and taking sufficiently small 
$\kappa_1$, $\kappa_2$ and $\delta$ depending on $n$, $p$, $q$ and $\epsilon$, we have the second estimate in \eqref{westimate2}.

Let $\tilde w_0\in W^{1,\phi}(B_2)$ be an even extension of $w_0$ so that $\tilde w_0(x)=w_0(x)$ if $x\in B_2^+$ and $\tilde w_0(x_1,\dots,x_{n-1},x_n)=w_0(x_1,\dots,x_{n-1},-x_n)$ if 
$(x_1,\dots,x_n)\in B_2^-$. Note that $\tilde w_0$ is well defined since $w_0=0$ on $B_r\cap\{x_n=0\}$. Moreover $\tilde w_0$ is a weak solution to
$$
\div \left(\frac{\phi'(|D\tilde w_0|)}{|D\tilde w_0|}D\tilde w_0\right)=0\quad\text{in}\ \ B_{2},
$$ 
see for instance \cite[Theorem 3.4]{Mart81}.
Therefore, \eqref{bdLip} follows from Lemma~\ref{lemHolder}.
\end{proof}


\section*{Acknowledgment}

J.\ Ok was supported by the National Research Foundation of Korea funded by the
Korean Government (NRF-2017R1C1B2010328) and P.\ H\"ast\"o was supported in part by the 
Magnus Ehrnrooth Foundation. We also thank the referee for useful comments.


\bibliographystyle{amsplain}

\end{document}